\documentclass{amsart}

\usepackage[english]{babel}

\usepackage[letterpaper,margin=1in]{geometry}

\usepackage{enumitem}

\usepackage{amsmath}
\usepackage{amssymb, amsfonts, amsthm}

\usepackage{tikz}
\usepackage{tikz-cd}
\usetikzlibrary{decorations.pathmorphing}

\usepackage{color}

\usepackage{enumitem}

\usepackage[normalem]{ulem}

\usepackage{graphicx}
\usepackage[colorlinks=true, allcolors=blue]{hyperref}
\usepackage{xcolor}

\newtheorem{theorem}{Theorem}[section]
\newtheorem{lemma}[theorem]{Lemma}
\newtheorem{proposition}[theorem]{Proposition}
\newtheorem{corollary}[theorem]{Corollary}

\theoremstyle{definition}
\newtheorem{definition}[theorem]{Definition}
\newtheorem{example}[theorem]{Example}

\theoremstyle{remark}
\newtheorem*{remark}{Remark}

\title{On Vanishing Theorems and Bogomolov's Inequality on Surfaces in Positive Characteristic}

\author{Fei Ye}
\address{Department of Mathematics and Computer Science, Queensborough Community College of CUNY, 222-05, 56th Avenue Bayside, NY 11364}
\address{
Department of Mathematics, CUNY Graduate Center, 365 Fifth Avenue, New York, NY 10016
}
\email{feye@qcc.cuny.edu}

\author{Zhixian Zhu}
\address{Department of Mathematics, Capital Normal University, Beijing China}

\email{zhixian@cnu.edu.cn}


\subjclass[2020]{}

\thanks{
The authors would like to thank Jiang Chen and Lei Zhang for invaluable discussion.
The first-named author is partially supported by the PSC-CUNY Award \# 68419-00 56.
}

\begin{document}

\begin{abstract}
In this paper, we study the equivalence between Bogomolov's instability theorem and the Miyaoka-Sakai theorem on surfaces in positive characteristic.
We show that Bogomolov's instability theorem can be derived from Miyaoka-Sakai theorem. Conversely, it implies a partial version of the Miyaoka-Sakai theorem that lacks the vanishing conclusion. This partial version is still sufficient to deduce the Mumford-Ramanujam vanishing theorem.

Additionally, we identify a class of surfaces in positive characteristic for which the Miyaoka-Sakai theorem (or a weaker variant), or the Kawamata-Viehweg vanishing theorem holds. In particular, we present a new proof of the Kawamata-Viehweg vanishing theorem on smooth del Pezzo surfaces.

As an application of the Miyaoka-Sakai theorem, we obtain Reider-type results concerning Fujita's conjecture.
\end{abstract}

\maketitle

\section{Introduction}

Over algebraic closed fields of characteristic zero, vanishing theorems have played irreplaceable roles in algebraic geometry, particularly, in the study of the minimal model program. However, after Raynaud's construction of the first counterexample to the Kodaira vanishing theorem in positive characteristic \cite{Raynaud1978}, additional counterexamples have been discovered for both Kodaira vanishing and Kawamata-Viehweg vanishing (see, for instance, \cite{Ekedahl1988, Xie2010, Xie2011, Mukai2013, Cascini2018, Totaro2019, Bernasconi2021}.  
Parallel to these developments, counterexamples have also been found for Bogomolov's instability theorem \cite{Raynaud1978, Mukai2013} and the Bogomolov-Miyaoka-Yau inequality \cite{Mukai2013}.
The pathology extends to Fujita's conjecture on adjoint linear systems \cite{Gu2022}, which has been established up to dimension five in characteristic zero (see \cite{Reider1988, Ein1993, Kawamata1997, Ye2020}). Despite the failure of Kodaira vanishing and Kawamata-Viehweg vanishing, and Bogomolov's inequality on many surfaces, particularly surfaces of general type or quasi-elliptic surfaces of Kodaira dimension one, substantial effort has been devoted to identifying conditions under which these theorems do hold  (see, for instance, \cite{Shepherd-Barron1991, Terakawa1998, Langer2022a, Koseki2023}).

In characteristic zero, it is known from \cite{DiCerbo2013, Fernandez1995} that Bogomolov's instability theorem is equivalent to both the Kawamata-Viehweg vanishing theorem and the following Miyaoka-Sakai theorem.

\begin{theorem}[Miyaoka-Sakai Theorem (see {\cite[Theorem 2.7]{Miyaoka1980}, \cite[Proposition 1]{Sakai1990}})]
Let $S$ be a smooth surface over an algebraically closed field of characteristic zero and $D$ be a big divisor such that $H^1(S, \mathcal{O}_S(-D))\ne 0$. There exists a nonzero effective divisor $B$ such that 
\begin{enumerate}
    \item $-(D - B)$ is nef on $B$, i.e. $(D-B)C \le 0$ for  any irreducible component $C$ of $B$.
    \item $H^1(S, \mathcal{O}_S(-D+B))=0$.
    \item $D-2B$ is big if $D^2>0$.
\end{enumerate}
\end{theorem}

Our primary interest lies in studying the equivalence of  Bogomolov's instability theorem, the Miyaoka-Sakai theorem, and the Kawamata-Viehweg vanishing theorem in positive characteristic. 
Following Sakai's approach \cite{Sakai1990}, one can deduce a partial version of the Miyaoka-Sakai theorem (Theorem \ref{thm:Effective_Sakai}) as well as Reider-type results (Theorem \ref{thm:ReiderType}) from the effective version of Bogomolov's instability theorem (Theorem \ref{thm:Koseki}). 
However, this method does not yield the vanishing of $H^1(S, \mathcal{O}_S(-D + B))$ directly, which is a crucial ingredient in Di Cerbo's proof \cite{DiCerbo2013} of Bogomolov's instability theorem via the Miyaoka-Sakai theorem.
The proof of this vanishing in \cite{DiCerbo2013} suggests that, in positive characteristic, a Kawamata-Viehweg-type vanishing theorem is necessary to establish the full Miyaoka-Sakai theorem. Indeed, it suffices to find a special subdivisor $Q\le D$ such that $H^1(S, \mathcal{O}_S(-Q))=0$ (see Lemma \ref{lem:Miyaoka}). 

Our second objective is to identify surfaces in positive characteristic on which the Miyaoka-Sakai theorem can be proved without relying on Bogomolov's instability theorem.

Recently, Enokizono proved a vanishing theorem \cite[Theorem 4.17]{Enokizono2023} for the integral Zariski decomposition of big divisors. A natural candidate for such a special subdivisor $Q\le D$ is the $\mathbb{Z}$-positive part $P_{\mathbb{Z}}$ of $D$, provided that $\dim H^0(S, \mathcal{O}_S(P_{\mathbb{Z}}))>\dim H^1(S, \mathcal{O}_S)_n$, where 
$$H^1(S, \mathcal{O}_S)_n = \ker\left(H^1(S,\mathcal{O}_S)\to H^1(S, F^e_*\mathcal{O}_S)\right), \quad e\gg 0$$ denotes the Frobenius nilpotent part of $H^1(S,\mathcal{O}_S)$. 

On smooth Frobenius split surfaces, Enokizono's condition is satisfied if $P_{\mathbb{Z}}$ is effective. 
Based on this observation, we show that Kawamata-Viehweg vanishing holds for effective, nef and big $\mathbb{Q}$-divisors on smooth Frobenius split surfaces (Proposition \ref{prop:D>=0_KV_FSplit}).

Without the effectiveness hypothesis of $P_{\mathbb{Z}}$, we prove that the Kawamata-Viehweg vanishing theorem continues to hold on Hirzebruch surfaces and smooth del Pezzo surfaces. Consequently, the Miyaoka-Sakai theorem also holds on them. Additionally, we establish the Miyaoka-Sakai theorem on Frobenius split ruled surfaces and smooth weak del Pezzo surfaces. Similar results can also be obtained on minimal surfaces of Kodaira dimension zero.

For smooth projective surfaces that are neither of general type nor quasi-elliptic surface of Kodaira dimension one,  inspired by Mukai's proof of the Mumford-Ramanujam vanishing theorem, that is, $H^1(S, \mathcal{O}(-D)) = 0$ if $D$ is nef and big (see \cite[the Theorem]{Ramanujam1974} and \cite[Section II]{Mumford1978}), we obtain a $\mathbb{Q}$-divisor version of the Miyaoka-Sakai Theorem. This version also implies the Mumford-Ramanujam vanishing theorem.

In summary, our investigation shows the following implications:
    $$\begin{array}{c} \text{Effective Bogomolov's instability theorem (Theorem \ref{thm:Koseki})} \\ \Downarrow \\ \text{Effective Miyaoka-Sakai theorem (Theorem \ref{thm:Effective_Sakai})} \\ \Downarrow \\ \text{Effective Ramanujam vanishing theorem (Corollary \ref{cor:NCVanishing})} \\ \Downarrow \\ \text{Effective Mumford-Ramanujam vanishing (Corollary \ref{cor:Effective_Mumford})} \end{array},$$
$$
\text{Kawamata-Viehweg vanishing theorem}
\Rightarrow
\text{Miyaoka-Sakai theorem}
\Rightarrow
\text{Bogomolov's instability theorem},
$$
and
$$\text{Weak Miyaoka-Sakai theorem (Corollary \ref{cor:weakMS})} \Longrightarrow \text{Mumford-Ramanujam vanishing (Corollary \ref{cor:MRVanishing})}.$$

The paper is organized as follows. In Section \ref{sec:preliminaries}, we fix some notations and review results on Bogomolov's instability theorem. 
In Section \ref{sec:Bogomolov-to-others}, we present implications of the effective Bogomolov's instability theorem (see Theorem \ref{thm:Koseki}), including Reider-type results. In Section \ref{sec:MiyaokaVanishingToBogmolov}, we show that the Miyaoka-Sakai vanishing theorem implies Bogomolov's instability theorem. In Section \ref{sec:vanishing_Kodaira<=0}, we show that the Kawamata-Viehweg vanishing theorem holds on Hirzebruch and del Pezzo surfaces, while the Miyaoka-Sakai theorem holds on more surfaces of Kodaira dimension at most zero. In Section \ref{sec:MiyaokaSakai=MumfordRamanujam}, we prove a $\mathbb{Q}$-divisor variant of the Miyaoka-Sakai theorem and show that it implies the Mumford-Ramanujam vanishing on surfaces that is neither of general type nor quasi-elliptic of Kodaira dimension one.

\section{Preliminaries}\label{sec:preliminaries}

Throughout the paper, we denote by $S$ a smooth projective surface over an algebraically closed field $\mathbf{k}$ of characteristic $p$ and by $K_S$ the canonical divisor on $S$. Unless explicitly indicated, a divisor means an integral Cartier divisor.

Bogomolov's instability theorem (\cite[10.12. Corollary 2]{Bogomolov1979}) has been extensively studied and generalized to positive characteristic, see \cite[Theorem 12]{Shepherd-Barron1991}, 
\cite[Theorem 7.1]{Langer2016},
\cite[Theorem 1.1]{Koseki2023}. We first recall the definition of slope stability.

\begin{definition}
Let $\mathcal{E}$ be a locally free sheaf of rank $r$ and $H$ a numerically non-trivial nef divisor on $S$. The \textbf{slope} of $\mathcal{E}$ with respect to $H$ is defined as
$$\mu_H(\mathcal{E})=\frac{c_1(\mathcal{E})\cdot H}{r}.$$

We say that $\mathcal{E}$ is \textbf{slope $H$-semistable} if
    $$\mu_H(\mathcal{F}) \le \mu_H(\mathcal{E})$$
    for any subsheaf $\mathcal{F}\subset \mathcal{E}$ with $0<\operatorname{rank}(\mathcal{F}) < \operatorname{rank}(\mathcal{E})$.  

We say $\mathcal{E}$ is \textbf{slope $H$-unstable} if it is not slope $H$-semistable.
\end{definition}

In positive characteristic, for Bogomolov's instablity theorem to hold on a surface of general type or a quasi-elliptic surface of Kodaria dimenion 1, it is known that a correction term in the inequality is needed.

Let $C_S$ be a nonnegative integer that is defined as follows:
\begin{itemize}
    \item if $\kappa(S)=2$ and $p > 2$, then
    $$C_S := \min\left\{\dfrac{\operatorname{Vol}(S)}{4}, 2 + 5 \operatorname{Vol}(S) -\chi(\mathcal{O}_{S})\right\},$$
    where $\operatorname{Vol}(S)=\limsup\limits_{m\to \infty}\frac{\dim H^0(S, \mathcal{O}_S(m K_S))}{m^2/2!}$ is the volume of $S$;
    \item if $\kappa(S)=2$ and $p = 2$, then
    $$
    C_S := \min\left\{\dfrac{\max\{\operatorname{Vol}(S), \operatorname{Vol}(S)-3\chi(\mathcal{O}_{S})+2\}}{4}, 2 + 5 \operatorname{Vol}(S) -\chi(\mathcal{O}_{S})\right\}; 
    $$
    \item if $\kappa(S)=1$, $S$ is quasi-elliptic, $p>0$, and $\chi(\mathcal{O}_S)\le 0$,
    then
    $$
    C_S := 2  - \chi(\mathcal{O}_S); 
    $$
    \item otherwise, $C_S = 0$,
\end{itemize}
Note that in the definition of $C_S$, $\operatorname{Vol}(S)=K_{S'}^2$, where $S'$ is the minimal model of $S$.

\begin{theorem}[{
\cite[10.12. Corollary 2]{Bogomolov1979}, 
\cite[Theorem 12]{Shepherd-Barron1991}, 
\cite[Theorem 7.1]{Langer2016},
\cite[Theorem 1.1]{Koseki2023}
}]\label{thm:Koseki}
    Let 
    $\mathcal{E}$ be a locally free sheaf of rank $2$ on $S$. If
     $$ c_1(\mathcal{E})^2 - 4 c_2(\mathcal{E}) > 
     4C_S,
     $$
     then $\mathcal{E}$ is slope $H$-unstable with respect to a certain numerically non-trivial nef divisor $H$. 
\end{theorem}

\section{From Bogomolov to Miyaoka-Sakai and then Mumform-Ramanujam\label{sec:Bogomolov-to-others}}

Following Mumford-Reider's method as seen in \cite[Section 3]{Sakai1990}, we obtain the following effective version of Miyaoka-Sakai's Theorem, which was first proved in \cite{Miyaoka1980} using the Miyaoka vanishing theorem and Zariski decomposition. A slightly different proof was also given by Sakai in \cite{Sakai1990}.

\begin{theorem}[Effective Miyaoka-Sakai Theorem]\label{thm:Effective_Sakai}
    Let $D$ be a big divisor on $S$ such that $D^2>4C_S$. Suppose that $H^1(S, \mathcal{O}_S(-D))\ne 0$. Then there exists a nonzero effective divisor $B$ such that
    \begin{enumerate}
        \item $(D-B)B\le 0$ \label{item:(D-B)B>=0}
        \item $D-2B$ is big and $(D- 2B)^2\ge D^2$. \label{item:D-2Bbig}
    \end{enumerate}
\end{theorem}
\begin{proof}
Since $H^1(S, \mathcal{O}_S(-D))\ne 0$, then there is a non-split exact sequence
    \begin{equation}\label{eq:non-split-seq}
        0\to \mathcal{O}_S \to \mathcal{E} \to \mathcal{O}_S(D)\to 0,
    \end{equation}
    where $\mathcal{E}$ is locally free sheaf of rank $2$ and $c_1(\mathcal{E}) = \mathcal{O}_S(D)$.
    Because
    $4c_2(\mathcal{E}) - c_1(\mathcal{E})^2 = D^2 > 4 C_S$,
    by Theorem \ref{thm:Koseki}, $\mathcal{E}$ is slope $H$-unstable with respect to a certain numerically non-trivial nef divisor $H$. 
    Then $\mathcal{E}$ has a maximal destabilizing subsheaf which is locally free of rank $1$. Let $A$ be the divisor associated to the maximal destabilizing subsheaf. 
    There exists an exact sequence
    \begin{equation}\label{eq:destablizing-seq}
    0\to \mathcal{O}_S(A) \to \mathcal{E}\to \mathcal{I}_Z\otimes \mathcal{O}_S(B)\to 0,
    \end{equation}
    where $B$ is also a divisor. 
    Moreover,
    \begin{equation}\label{eq:destable-ineq}
    (A - B) H = 2(\mu_H(\mathcal{O}_S(A)) - \mu_H(\mathcal{E})) \ge 0. 
    \end{equation}

By twisting the short exact sequence \eqref{eq:destablizing-seq} by $\mathcal{O}_S(-A)$, we see that $c_2(\mathcal{E}(-A)) = \deg Z \ge 0$. On the other hand, we have
    \begin{equation}\label{eq:AB<0}
        \begin{aligned}
            c_2(\mathcal{E}(-A)) = &c_2(\mathcal{E}) + c_1(\mathcal{E})c_1(\mathcal{O}_S(-A)) + c_1(\mathcal{O}_S(-A))^2\\
            = & -A(A + B)+(-A)^2\\
            = &- AB \\
            = & -(D - B)B.
        \end{aligned}
    \end{equation}
Comparing the calculations for $c_2(\mathcal{E})$, we see that $(D - B)B = AB = -\deg Z \le 0$ which confirms part \eqref{item:(D-B)B>=0}. Moreover, this inequality also implies that
    \begin{equation}
        \label{eq:(A-B)^2>0}
    (A - B)^2 = (A + B)^2 - 4AB = D^2 - 4AB \ge D^2 > 4C_S >0. 
    \end{equation}

If $(A  - B)H = 0$, then by the Hodge index theorem (see \cite[Corollary 2.4]{Badescu2001}) and the nefness of $H$, we see that $H^2 = 0$ and $H$ is numerically trivial. This contradicts the assumption on $H$.
Therefore, $(A - B)H > 0$. By perturbing $H$ a little bit, we may assume that $H$ is ample. By \cite[Proposition 2.2]{Badescu2001}, we know that $A-B=D-2B$ is a big divisor. Part \eqref{item:D-2Bbig} is concluded.

To show that $B$ is effective and non-trivial, consider the following commutative diagram
    $$
     \begin{tikzcd}[row sep=1.5em, column sep=1em]
        & & 0 \arrow[d] & & \\
        & & \mathcal{O}_S \arrow[d, "p"] \arrow[dr, "\beta=g\circ p"] & & \\
        0 \arrow[r] & \mathcal{O}_S(A) \arrow[dr, "\alpha=q\circ f"'] \arrow[r, "f"] & \mathcal{E} \arrow[r, "g"] \arrow[d, "q"] & \mathcal{I}_Z \otimes \mathcal{O}_S(B) \arrow[r] & 0 \\
        & & \mathcal{O}_S(D) \arrow[d] & & \\
        & & 0 & &
    \end{tikzcd}.
    $$    
We claim that the composition $\alpha: \mathcal{O}_S(A)\overset{f}{\to} \mathcal{E}\overset{q}{\to} \mathcal{O}_S(D)$ is nontrivial. Assume on the contrary that $q \circ f$ is trivial. Because the vertical sequence is exact, there exists a nontrivial morphism $\phi: O_S(A)\to \mathcal{O}_S$.
    It follows that $-A$ is effective, and hence $(A - B)H = (2A - D)H=2AH - DH \le 0$ which contradicts to that $(A - B)H > 0$.
    Therefore, $\alpha$ is the natural inclusion and $B = D - A$ is effective.
    
    We claim that $B\ne 0$. Otherwise, we have $A=D$, $\deg Z =0$, and the sequence \eqref{eq:destablizing-seq} becomes
    $$0\to \mathcal{O}_S(D) \to \mathcal{E} \to \mathcal{O}_S\to 0$$
    which splits the exact sequence \eqref{eq:non-split-seq}. That contradicts the assumption that the sequence \eqref{eq:non-split-seq} is non-split.
\end{proof}

Theorem \ref{thm:Effective_Sakai} leads to an effective version of the Ramanujam vanishing theorem on numerically connected big divisors which implies an effective version of Mumford-Ramanujam vanishing theorem.

A pseudoeffective divisor $D$ on a smooth surface $S$ is \textbf{numerically connected} if $AB>0$ for any nontrivial pseudoeffective decomposition $D = A + B$. The set of numerically connected divisors includes nef and big divisors.

\begin{corollary}[Effective Ramanujam Vanishing Theorem]\label{cor:NCVanishing}
   Let $D$ be a pseudoeffective divisor on $S$ such that $D^2>4C_S$. If $D$ is numerically connected, then $H^1(S, \mathcal{O}_S(-D))=0$.
\end{corollary}
\begin{proof}
    By Zariski decomposition, $D$ is big. If $H^1(S, \mathcal{O}_S(-D))\ne 0$, by Theorem \ref{thm:Effective_Sakai}, there would exist a nonzero decomposition $D=A+B$ of pseudoeffective divisors such that $AB\le 0$, which contradicts the assumption that $D$ is numerically connected.
\end{proof}

\begin{lemma}[{Ramanujam's Connectedness Lemma (see \cite[Proof of Theorem]{Ramanujam1974}, \cite[Lemma 1.3]{Kawachi1998}, or \cite[Lemma 2]{Sakai1990})}]\label{lem:Nef=>NumCon}
Let $D$ be a nef and big divisor on $S$ and $D = M + N$ be a decomposition of pseudo-effective divisors. If $MN\le 0$, then either $M$ or $N$ is numerically trivial.
\end{lemma}

\begin{corollary}[Effective Mumford-Ramanujam Vanishing Theorem]\label{cor:Effective_Mumford}
    Let $D$ be a nef divisor on $S$ such that $D^2>4C_S$. Then $H^1(S, \mathcal{O}_S(-D))=0$.
\end{corollary}
\begin{proof}
Since $D$ is nef and $D^2>0$, by Lemma \ref{lem:Nef=>NumCon}, $D$ is numerically connected. The conclusion then follows from Corollary \ref{cor:NCVanishing}.
\end{proof}

\begin{remark}
Effective versions of Mumford-Ramanujam vanishing theorem have been presented in several papers under certain conditions, see for example,
\cite[Theorem 1.6.]{Terakawa1999},
\cite[Theorem 3]{Mukai2013},
\cite[Corollary 5.9]{DiCerbo2015}, and
    \cite[Corollary 7.4]{Langer2016}.
\end{remark}

As another application of Theorem \ref{thm:Effective_Sakai}, we present Reider-type results for adjoint divisors in positive characteristic. 
We note that Theorem \ref{thm:ReiderType} is not entirely new. Related results have been studied in several works that include \cite{Larsen2025}, \cite{Enokizono2023}, \cite{Terakawa1999}, \cite{Nakashima1993b}. Here, we present a cohomological approach following the method of Sakai (\cite[Theorem 3]{Sakai1990}).

\begin{theorem}\label{thm:ReiderType}
    Let $D$ be a nef divisor on $S$.
    \begin{enumerate}
        \item If $D^2\ge 4C_S + 5$, then $K_S+D$ is basepoint-free unless there exists a divisor $B$ such that either
       $DB=1$ and $B^2=0$, or $DB=0$ and $B^2=-1$.
       \item If $D^2\ge 4C_S + 9$, then $K_S+D$ is very ample unless there exists a divisor $B$ such that one of the following condition holds:
    \begin{itemize}
    \item $DB = 0$, $B^2=-2$ or $B^2= -1$;
    \item $DB = 1$, $B^2=-1$ or $B^2 = 0$;
    \item $DB = 2$, $B^2=0$;
    \item $C_S=0$, $D^2=9$, and $D\equiv 3B$.
    \end{itemize}
    \end{enumerate}
\end{theorem}

 \begin{proof}

In this proof, for any subscheme $Z\subset S$, we denote its ideal sheaf by $\mathcal{I}_Z$.

To prove the theorem, we consider the following cases.

\begin{enumerate}
    \item 
Suppose that $K_S + D$ is not free at $Z =\{p\}$ where $p$ is a point. Then $H^1(S, \mathcal{O}_S(K_S+D)\otimes_{\mathcal{O}_S}\mathcal{I}_Z)\ne 0$. 
Let $\phi: \tilde{S}\to S$ be the blowup at $p$ and $E$ the exceptional divisor. It follows that $\phi_*(\mathcal{O}_{\tilde{S}}(-E)) = \mathcal{I}_Z$, $E^2 = -1$, and 
$K_{\tilde{S}} =\phi^*(K_S) + E$.

\item Suppose that $K_S + D$ does not separate two distinct points $p$ and $q$. Then $H^1(S, \mathcal{O}_S(K_S+D)\otimes_{\mathcal{O}_S}\mathcal{I}_Z)\ne 0$, where $Z$ is the reduced closed subscheme supported on $\{p, q\}$. 
Let $\phi: \tilde{S}\to S$ be the blowup of $S$ along $Z$, and $E$ be exceptional divisor over $Z$. It follows that
$\phi_*\mathcal{O}_{\tilde{S}}(-E) = \mathcal{I}_Z$, $E^2=-2$, and $K_{\tilde{S}} = \phi^*(K_S) + E$.

\item If $K_S + D$ does not separate tangents at a point $p$, then $H^0(S, \mathcal{O}_S(K_S + D))\to \mathcal{O}_S(K_S + D)\otimes_{\mathcal{O}_S}\mathcal{O}_S/\mathfrak{m}_p^2$ is not surjective. 
Let $(x, y)$ be local coordinates of $S$ at $p$, $Z$ be the 0-dimensional subscheme supported at $p$ and locally defined by $(x, y^2)$. 
Consequently, $H^0(S, \mathcal{O}_S(K_S + D))\to \mathcal{O}_S(K_S + D)\otimes_{\mathcal{O}_S}\mathcal{O}_S/\mathcal{I}_Z$ is not surjective. 
Then $H^1(S, \mathcal{O}_S(K_S+D)\otimes_{\mathcal{O}_S}\mathcal{I}_Z)\ne 0$. 
Let $\phi_1: S_1\to S$ be the blowup of $S$ at $p$ with the exceptional divisor $E_1$ over $p$, and $p'$ the intersection point of $E_1$ with the strict transform of the curve locally defined $y=0$. 
Let $\phi_2: \tilde{S}\to S_1$ be the blowup of $S_1$ at the 
point $p'$ and $E_2$ be the exceptional divisor over $p'$.
Let $\phi=\phi_2\circ\phi_1: \tilde{S}\to S$ be the composition and $E = \bar{E}_1 + 2E_2$ where $\bar{E}_1$ be the strict transform of $E_1$. Then $\phi_*\mathcal{O}_{\tilde{S}}(-E) = \mathcal{I}_Z$, $E^2 = -2$, $K_{\tilde{S}} =\phi^*K_S + E$.
\end{enumerate}

Let $\widetilde{D} = \phi^*D - 2E$. Then we have
$$
H^1(\tilde{S}, \mathcal{O}_{\tilde{S}}(-\widetilde{D})) = H^1(S, \phi_* \mathcal{O}_{\tilde{S}}(\phi^*K_S+\phi^*D - E)) = H^1(S, \mathcal{O}_S(K_S + D)\otimes_{\mathcal{O}_S}\mathcal{I}_Z)\ne 0.
$$

 From the definition of $E$, we see that $\widetilde{D}^2=(\phi^*D - 2E)^2 = D^2 - 4\deg Z> 4 C_S$, where $\deg Z = -E^2$.
 By Theorem \ref{thm:Effective_Sakai}, there exists a nonzero effective divisor $\widetilde{B}$ such that $(\widetilde{D}-\widetilde{B})\widetilde{B}<0$ and $(\widetilde{D}-2\widetilde{B})$ is big. Set $B=\phi_*\widetilde{B}$. Note that $B$ is effective because $\widetilde{B}$ is effective. Then there exists a divisor $G$ supported in the exceptional divisor $E$ such that $\phi^*B = \widetilde{B} + G$.
    It follows that
    $$\begin{aligned}
        (\widetilde{D} - \widetilde{B})\widetilde{B} = & (\phi^*D -\phi^*B - 2E + G )(\phi^*B - G)\\
        = & (D-B)B + (2E-G)G\\
        = & (D-B)B + E^2 - (E-G)^2\\
        \ge & (D-B)B + E^2.
    \end{aligned}$$
    Since $(\widetilde{D}-\widetilde{B})\widetilde{B}\le 0$, then $(D-B)B\le -E^2=\deg Z$, or equivalently, $B^2\ge DB-\deg Z$. 
    
    Since $(\widetilde{D}-2\widetilde{B})$ is big, by \cite[Lemma 3]{Sakai1990}, $D-2B$ is big. Since $D$ is nef and big, by Lemma \ref{lem:Nef=>NumCon}, we know that $(D-2B)2B>0$, which implies that $B^2<\frac{DB}{2}$. Therefore, $DB/2> DB-\deg Z$. Consequently, $$DB-\deg Z\leq B^2<\dfrac{DB}{2}<\deg Z.$$ 

From the definition of $Z$, we obtain the following conclusions.

If $K_S + D$ is not basepoint-free, then $\deg Z = -E^2 = 1$ which yields $0\le DB\le 1$ and $DB-1\le B^2< \frac{DB}{2}$.  If $DB=0$, then $B^2 = -1$. If $DB=1$, then $B^2=0$.

If $K_S+D$ is not very ample, then $\deg Z = - E^2 = 2$ which yields $0\le DB\le 3$ and $DB-2\le B^2 < \frac{DB}{2}$.
If $DB=0$, then $B^2=-1$ or $B^2=-2$.
If $DB = 1$, then $B^2 = 0$ or $B^2 = -1$.
If $DB=2$, then $B^2 = 0$.
If $DB = 3$, then $B^2=1$. In this last case, by the Hodge index theorem (see \cite[Corollary 2.4]{Badescu2001} and the assumption on $D$, we see that 
$$0\ge (D-3B)^2 = D^2 - 6DB + 9 B^2 =  D^2 - 9 \ge 4C_S \ge 0.$$
Consequently, $C_S=0$ and $(D-3B)^2=0$ which implies that $D$ is numerically equivalent to $3B$ by the Hodge index theorem.
\end{proof}

Because $C_S=0$ when $\kappa(S)\le 1$ and $S$ is not quasi-elliptic with $\kappa(S)=1$, Fujita's conjecture holds true on such a surface. When $S$ is quasi-elliptic with $\kappa(S)=1$, Fujita's conjecture has been confirmed in \cite{Chen2021}.
When $\kappa(S) =2$, in \cite{Gu2022}, Gu, Zhang, and Zhang show that for each nonnegative integer $m$, there is a surface $S$ and a divisor $A_R$ such that $\mathcal{O}_S(K_S + m A_R)$ is not globally generated. Consequently, Fujita's conjecture fails for $(S, A_R)$. By straightforward calculations, one can check that $(mA_R)^2 < 4C_S + 5$, that is, $m A_R$ does not satisfy the condition of Theorem \ref{thm:ReiderType}.

A generalization of Fujita's conjecture concerning the $k$-very ampleness of adjoint line bundle on quasi-elliptic surfaces with Kodaira dimension $1$ appears in \cite{Zhang2023}. A study of $k$-very ampleness in a general setting can be found in \cite{Terakawa1999}. It is worth noting that the argument employed in the proof of Theorem \ref{thm:ReiderType} can be generalized to establish similar results on $k$-spannedness. By applying Bridgeland stability conditions, Larsen and Tenie \cite{Larsen2025} have recently obtained Kodaira-type vanishing theorems and Reider-type results on normal projective surfaces.

\section{From Miyaoka-Sakai to Bogomolov\label{sec:MiyaokaVanishingToBogmolov}}

In characteristic zero, Theorem \ref{thm:Effective_Sakai} follows from Zariski decomposition and the Kawamata-Viehweg vanishing theorem (see \cite[Section 3]{Sakai1990}). Additionally, this approach also leads to the vanishing $H^1(S, \mathcal{O}_S(-D + B))=0$ (see \cite[Corollary 2.1]{DiCerbo2013}),
which plays a crucial role in deducing Bogomolov's instability theorem (see \cite{DiCerbo2013}) from the Miyaoka-Sakai theorem.

The situation differs in positive characteristic, where the Kawamata-Viehweg vanishing theorem is known to fail in general (see for example \cite[Theorem 3.1]{Cascini2018}). In particular, it is unclear that whether the vanishing $H^1(S, \mathcal{O}_S(-D + B)) = 0$ continue to hold. Nevertheless, given this vanishing, one can deduce Bogomolov's instability theorem (see Proposition \ref{prop:SakaiToBogomolov}). This observation motivates the following definition, which will simplify our subsequent discussion.

\begin{definition}
    Let $D$ be a big divisor $S$ such that $D^2>0$. We say that $D$ is a \textbf{Miyaoka-Sakai divisor} if there exists an effective divisor $B$ such that
    \begin{enumerate}
        \item if $B\ne 0$, then $(D - B)C\le 0$ for any irreducible components $C$ of $B$.
        \item $D-2B$ is big and $(D - 2B)^2 \ge D^2$.
        \item $H^1(S, \mathcal{O}_S(-D + B))= 0$.
    \end{enumerate}
\end{definition}

We say that the \textbf{Miyaoka-Sakai theorem holds on $S$} if every big divisor $D$ with $D^2>0$ on $S$ is a Miyaoka-Sakai divisor.

For any big divisor $D$ with $D^2>0$ on $S$, it is unclear whether it is a Miyaoka-Sakai divisor. 
One sufficient condition for $D$ being a Miyaoka-Sakai divisor is the existence of a divisor $Q$ such that $H^1(S, \mathcal{O}_S(-Q))=0$ and $D-Q$ is a subdivisor of the negative part of the Zariski decomposition of $D$.

\begin{lemma}\label{lem:Miyaoka}
    Let $D$ be a big divisor on $S$ such that $D^2>0$ and $H^1(S,\mathcal{O}_S(-D))\ne 0$. If there exists a divisor $Q \le D$ such that $H^1(S, \mathcal{O}_S(-Q))=0$ and $D - Q$ is a subdivisor of the negative part of the Zariski decomposition of $D$, then $D$ is a Miyaoka-Sakai divisor.
\end{lemma}
\begin{proof}
    Let $M = D - Q$
    which is a nonzero effective divisor by the definition of $Q$ and the assumption on $D$. 
    Write $M = \sum\limits_{i=1}^r k_i C_i$ where $k_i>0$. Consider the set of $r$-tuples
    $$
    T:=\{ (j_1, \cdots , j_r)\in \mathbb{Z}^r \mid 0\le j_i\le k_i, H^1(S, \mathcal{O}_S(-D +\sum_{i=1}^r j_i C_i)=0\}.
    $$
    Since $H^1(S, \mathcal{O}_S(-D + \sum\limits_{i=1}^r k_i C_i))=H^1(S, \mathcal{O}_S(-Q))=0$, the set $T$ is nonempty. Moreover, it is bounded because $0\le j_i\le k_i$. 
    Therefore, $T$ has a minimal element $(m_1, \cdots, m_r)$ with respect to the natural partial ordering: $(j_1, \cdots, j_r) < (j_1', \cdots, j_r')$ if $j_i\le j_i'$ 
    and $j_a < j_a'$ for all $1\le i\le r$ and some $1\le a\le r$. 
    Let $B = \sum\limits_{i=1}^r m_i C_i$. Then $H^1(S, \mathcal{O}_S(-D+B))=0$. 
    By the minimality of $B$, we know that $H^1(S, \mathcal{O}_S(- D + B - C_i))\ne 0$ for any $i$ such that $m_i>0$.
    From the short exact sequence
    $$
    0\to \mathcal{O}_S(- D + B - C_i)\to \mathcal{O}_S(- D + B)\to \mathcal{O}_S(- D + B)|_{C_i}\to 0,
    $$
    we see that 
    $$
    \dim H^0(C_i, \mathcal{O}_S(- D + B)|_{C_i})\ge \dim H^1(S, \mathcal{O}_S(- D + B - C_i)) \ge 1.
    $$
    If follows that $\mathcal{O}_S(- D + B)|_{C_i}$ is effective, which implies that $(D - B)C_i \le 0$.

    Since $B$ is effective, then $(D - B)B\le 0$. Therefore,
    $$
    (D - 2B)^2 = D^2 - 4DB + 4B^2 = D^2 - 4B(D-B)\ge D^2 > 0.
    $$
Let $P$ be the positive part of the Zariski decomposition of $D$. Then $PB = 0$ and $PD = P^2 >0$ which implies that $(D-2B)P > 0$. Therefore, $D - 2B$ is big.
\end{proof}

Clearly, such a divisor $Q$ exists if Kawamata-Viehweg vanishing theorem holds on $S$. 
In general, the existence of $Q\le D$ with $H^1(S, \mathcal{O}_S(-Q))=0$ usually requires stronger positivity of $D$.

\begin{corollary}
    Let $D$ be a big divisor on $S$ such that $D^2 > 0$. There exists an integer $m>0$ such that $kmD$ is a Miyaoka-Sakai divisor for all positive integer $k\ge 0$. 
\end{corollary}
\begin{proof}
    Let $D = P  +  N$ be the Zariski decomposition. Then $P^2\ge D^2 >0$. Let $m$ be the smallest integer such that $mP$ is integral and $(mP)^2 > 4 C_S$. The $kmP$ is integral and $(kmP)^2 > 4C_S$. Let $Q=kmP$. By Corollary \ref{cor:Effective_Mumford}, $H^1(S, \mathcal{O}_S(-Q))=0$. By Lemma \ref{lem:Miyaoka}, $kmD$ is a Miyaoka-Sakai divisor.
\end{proof}

It was proved by Di Cerbo \cite{DiCerbo2013} that Miyaoka-Sakai Theorem (\cite[Proposition 1]{Sakai1984}) is equivalent to Bogomolov instability theorem (\cite[10.12. Corollary 2]{Bogomolov1979}) in characteristic $0$. 
Moreover, by \cite{Miyaoka1980}, \cite{Sakai1984} and \cite{Fernandez1995}, these two theorems are also equivalent to Kawamata-Viehweg vanishing theorem.

In this section, we show that Bogomolov instability theorem holds on surfaces where Miyaoka-Sakai vanishing holds.

We first recall the following lemma due to Fern{\'a}denz del Busto which is still valid in positive characteristic.  
\begin{lemma}[{\cite[Uniform Multiplicity Property (UMP)]{Fernandez1995, Lazarsfeld1997}}]\label{lem:UMP}
Let $\mathcal{E}$ be a globally generated locally free sheaf of rank $2$ on $S$, 
$P=\{[s] \mid \dim Z(s) = 0\}\subset \mathbb{P}H^0(S, \mathcal{E})$, and $\mathfrak{Z}=\{(x, [s])\mid x\in Z(s)\}\subset S\times P$. Let $\mathfrak{D}\supset \mathfrak{Z}$ be an effective divisor on $S\times P$ flat over $P$. Suppose that for every general $[s]\in P$, there is a point $x\in Z(s)$ and a positive integer $m$ such that  $\mathrm{mult}_x D_s \ge m$ for $D_s =\mathfrak{D}|_s$.
Then $\mathrm{mult}_x D_s \ge m$ for a general section $[s]\in P$ and for all $x\in Z(s)$.
\end{lemma}

\begin{proposition}\label{prop:SakaiToBogomolov}
Assume that Miyaoka-Sakai vanishing holds on $S$. Then a locally free sheaf $\mathcal{E}$ of rank $2$ on $S$ with $\Delta = c_1(\mathcal{E})^2-4c_2(\mathcal{E})>0$ is slope $H$-unstable with respect to a certain ample divisor $H$.
\end{proposition}

\begin{proof}
    Let $\mathcal{E}$ be a locally free sheaf of rank $2$ on $S$ and $D$ be a divisor such that $\det(\mathcal{E})=\mathcal{O}_S(D)$. 
    For any divisor $G$ on $S$, we have
    $$c_1(\mathcal{E}\otimes\mathcal{O}_S(G)) = D + 2G\quad \text{and}\quad c_2(\mathcal{E}\otimes \mathcal{O}_S(G)) = DG + G^2 + c_2(\mathcal{E}).$$
    It follows that 
    $$c_1(\mathcal{E})^2 - 4 c_2(\mathcal{E}) =c_1(\mathcal{E}\otimes_{\mathcal{O}_S}\mathcal{O}_S(G))^2 - 4c_2(\mathcal{E}\otimes_{\mathcal{O}_S}\mathcal{O}_S(G))$$ for any divisor $G$.
    
    Since twisting a line bundle does not change the stability and Bogomolov's inequality, by twisting a sufficiently ample divisor, we may assume that $D$ is a very ample, $c_2(\mathcal{E})>0$, and $\mathcal{E}$ is globally generated, and the zero locus $Z_s$ of a general section $s$ of $\mathcal{E}$ consists of $r = c_2(\mathcal{E})$ distinct points by Bertini's theorem (see \cite[Proposition A]{Spreafico1996} or \cite[Proposition 11]{Kleiman1974}).
    Consider the Koszul complex determined by a general section $s$ which yields a short exact sequence
    $$0\to \mathcal{O}_S\to \mathcal{E} \to \mathcal{O}_S(D) \otimes\mathcal{I}_{Z_s}\to 0,$$
    where $\mathcal{I}_{Z_s}$ is the ideal sheaf of $Z_s$. 
    Because both $\mathcal{E}$ and $\mathcal{O}_S$ are locally free, the exact sequence is non-split. 
    It follows that $H^1(S, \mathcal{O}_S(K_S + D)\otimes\mathcal{I}_{Z_s})\ne 0$.
    Let $\pi: \widetilde{S}\to S$ be the blowup of $S$ along $Z_s$ and $E = E_1 +\cdots + E_r$ be the exceptional divisor. 
    Then
        $$H^1(\widetilde{S}, \mathcal{O}_{\widetilde{S}}(-\pi^*D + 2E)) = H^1(\widetilde{S}, \mathcal{O}_{\widetilde{S}}(K_{\widetilde{S}} + \pi^*D - 2E)) = H^1(S, \mathcal{O}_S(K_S + D)\otimes\mathcal{I}_{Z_s}) \ne 0.$$
     Let $\widetilde{D} = \pi^*D - 2E$. We see that $$\widetilde{D}^2 = D^2 - 4 r = c_1(\mathcal{E})- 4c_2(\mathcal{E}) > 0,$$ 
    and $\pi_*\widetilde{D} = D$ is big.
    By \cite[Lemma 3]{Sakai1990}, $\widetilde{D}$ is also big.
    Since Miyaoka-Sakai vanishing holds on $S$, there exists nonzero effective divisor $\widetilde{B}_s$ 
    such that $(\widetilde{D}_S - \widetilde{B}_s)\widetilde{B}_s\le 0$ and $H^1(\widetilde{S}, \mathcal{O}_{\widetilde{S}}(- \widetilde{D} + \widetilde{B}_s))=0$. 
    Note that $\widetilde{B}_s$ depends on $E$ hence the section $s$.
    Let $B_s = \pi_*\widetilde{B}_s$. Write $\widetilde{B}_s = \pi^*B_s + \sum\limits_{i=1}^r a_i E_i$, where $a_i$ are integers. We want to show that there exists $i$ such that $a_i < 0$, and hence $B_s$ passes through a point $x\in Z_s$.
    
    From the inequality $(\widetilde{D} - \widetilde{B})\widetilde{B} \le 0$, we obtain
    $$(D - B_s) B_s + \sum a_i(2+a_i) = \big((\pi^*D-2E) - (\pi^*B + \sum a_i E_i)\big)(\pi^*B + \sum a_i E_i) = (\widetilde{D} - \widetilde{B})\widetilde{B}\le 0.$$
    Because $D$ is ample, $B_s$ is effective, $D-2B_s$ is big, we see that $(D-B_s)B_s > 0$ by Lemma \ref{lem:Nef=>NumCon}.
    Consequently, $\sum\limits_{i=1}^r a_i(2+a_i) < 0$ which implies that there exists an index $i$ such that $a_i<0$.

    Therefore, $\mathrm{mult}_xB_s > 0$ for some $x\in Z_s$. By the uniform multiplicity property (Lemma \ref{lem:UMP}), by taking a general $s$, we may assume that $\mathrm{mult}_xB_s>0$ for all $x\in Z_s$. It follows that $\mathcal{I}_{Z_s}\supset \mathcal{O}_S(-B_s)$ and  $\beta: \mathcal{O}_S(D-B_s)\to \mathcal{O}_S(D)\otimes\mathcal{I}_{Z_s}$ is injective, which induces the following commutative diagram of extensions
        $$\begin{tikzcd}
        0\arrow[r] & \mathcal{O}_S\arrow[r]\arrow[d, equal] & \mathcal{F}\arrow[r]\arrow[d, hook]& \mathcal{O}_S(D-B_s)\arrow[r] \arrow[d, hook, "\beta"] & 0\\
        0\arrow[r] & \mathcal{O}_S\arrow[r] & \mathcal{E}\arrow[r] & \mathcal{I}_{Z_s}\otimes \mathcal{O}_S(D)\arrow[r] &0.
    \end{tikzcd}$$
        
    Since $H^1(\widetilde{S}, \mathcal{O}_{\widetilde{S}}(- \widetilde{D} + \widetilde{B}_s))=0$, by the projection formula and Leray spectral sequence, we know that $H^1(S, \mathcal{O}_S( - D + B_s)\otimes \mathcal{I}_{Z'})=0$, where $\mathcal{I}_{Z'} =\pi_*\mathcal{O}_S(-(2+a_i)E_i)$. Since $Z'$ has dimension $0$, from the short exact sequence
    $0\to \mathcal{O}_S(- D + B_s)\otimes \mathcal{I}_{Z'} \to \mathcal{O}_S(- D + B_s) \to \mathcal{O}_S(- D + B_s)\otimes \mathcal{O}_S/\mathcal{I}_{Z'}\to 0$,
    taking cohomology, we obtain
   $$0 = H^1(S, \mathcal{O}_S(- D + B_s))=\mathrm{Ext}^1(\mathcal{O}_S, \mathcal{O}_S(- D + B_s)) = \mathrm{Ext}^1(\mathcal{O}_S(D - B_s), \mathcal{O}_S).$$
    It follows that any extension of $\mathcal{O}_S(D - B_s)$ by $\mathcal{O}_S$ is splitting. Therefore, the injection $\beta$ lifts to an injection $\mathcal{O}_S(D-B_s)\to \mathcal{E}$.
 
    Since $D-2B_s$ is big, $(D-2B_s)^2 \ge D^2 >0$, and $D$ is ample, by Hodge index theorem, we know that $(D-2B_s)D>0$ which implies that 
    $$\mu_D(D - B_s) = (D - B_s) D > \frac{1}{2}(D^2 + (D - 2B_s)D) > \frac{D^2}{2} = \mu_D(\mathcal{E}).$$ Therefore, $\mathcal{E}$ is slope $D$-unstable. This completes the proof.
\end{proof}

\section{Vanishing Theorems on Surfaces of Kodaira Dimension at Most Zero\label{sec:vanishing_Kodaira<=0}}

In \cite{Enokizono2023}, Enokizono presents an effective version of the Kawamata-Viehweg vanishing theorem on normal surfaces in positive characteristic. Assume that $p>0$. Applying this result, we show that the Kawamata-Viehweg vanishing theorem holds on Hirzebruch and del Pezzo surfaces, while the Miyakao-Sakai vanishing theorem holds in addition on smooth Frobenius split ruled surfaces, weak del Pezzo surfaces, Abelian surfaces, and hyperelliptic surfaces. 

\begin{definition}[{\cite[Definition 3.10]{Enokizono2023}}]
A divisor $D$ on $S$ is called $\mathbb{Z}$-positive if $B - D$ is not nef over $B$ (i.e., $BC < DC$ holds for some irreducible component $C$ of $B$) for any effective negative definite divisor $B > 0$ on $S$.
\end{definition}

Enokizono (\cite[Theorem 3.5]{Enokizono2024}) shows that a pseudoeffective divisor has a unique integral Zariski decomposition whose positive part is a $\mathbb{Z}$-positive integral divisor.

\begin{theorem}[Integral Zariski Decomposition]
   Let $D$ be a pseudoeffective divisor on $S$. Then there exists a unique decomposition $D = P_{\mathbb{Z}} + N_{\mathbb{Z}}$ such that
   \begin{enumerate}
       \item $P_{\mathbb{Z}}$ is $\mathbb{Z}$-positive.
       \item $N_{\mathbb{Z}} = 0$ or $N_{\mathbb{Z}} > 0$ is a negative definite exceptional integral divisor.
       \item $P_{\mathbb{Z}} C \le 0$ for any irreducible component of $N_{\mathbb{Z}}$.
   \end{enumerate}
\end{theorem}

 According to \cite[Proposition 3.16 and Remark 3.17]{Enokizono2024}, the integral Zariski decomposition of a pseudo-effective $D$ can be obtained from the standard Zariski decomposition $D = P + N$ via the following iterative procedure. 
Set $P_0 = \lceil P \rceil$, the roundup of $P$. 
For $i \ge 1$, let $N_{i-1} = D - P_{i-1}$. If there exists an irreducible component $C_i$ of $N_{i-1}$ such that $P_{i-1} \cdot C_i > 0$, set $P_i = P_{i - 1} + C_i$. The process terminates when $P_i C\le 0$ for any irreducible component of $N_i = D - P_i$. 

Consequently, we have the following characterization of the $\mathbb{Z}$-positive part of the integral Zariski decomposition.

\begin{proposition}\label{prop:IntZD_from_ZD}
    Let $D$ be a pseudo-effective divisor on $S$ with the Zariski decomposition $D = P + N$, where $P$ is the positive part. Then the $\mathbb{Z}$-positive part of $P_{\mathbb{Z}}$ of the integral Zariski decomposition of $D$ can be written as $P_{\mathbb{Z}} =  P + M$, where $M$ is an effective $\mathbb{Q}$-divisor satisfying $M\le N$.
\end{proposition}

By \cite[Proposition 2.3.21]{Lazarsfeld2004}, we see that the following lemma follows immediately from Proposition \ref{prop:IntZD_from_ZD}.

\begin{lemma}\label{lem:h0D=h0P}
    Let $D = P_{\mathbb{Z}} + N_{\mathbb{Z}}$ be the integral Zariski decomposition. Then $H^0(S, \mathcal{O}_S(D))=H^0(S,\mathcal{O}_S(P_{\mathbb{Z}}))$.
\end{lemma}

Let $F: S\to S$ be the Frobenius morphism. Denote by $H^1(S, \mathcal{O}_S)_n$ the \textbf{Frobenius-nilpotent} part of $H^1(S, \mathcal{O}_S)$, that is,
$$H^1(S, \mathcal{O}_S)_n = \ker((F^e)^*: H^1(S, \mathcal{O}_S) \to H^1(S, \mathcal{O}_S))$$
for a sufficiently large $e$, where $(F^e)^*=(F^*)^e$ is the $e$-iteration of the $p$-linear map $F^*(a v) = a^p F^*(v)$ for any $a\in \mathbf{k}$ and $v\in H^1(S, \mathcal{O}_S)$ induced by $F$.

 Another important observation of Enokizono is a vanishing theorem for the positive part of the integral Zariski decomposition of a big divisor.
 
\begin{theorem}[Enokizono Vanishing {\cite[Theorem 4.17]{Enokizono2023}}]\label{thm:Enokizono}
    Let $D$ be a big $\mathbb{Z}$-positive divisor on $S$. If $\dim H^0(S, \mathcal{O}_S(D)) > \dim H^1(S, \mathcal{O}_S)_n$, then
    $H^1(S, \mathcal{O}_S(-D))=0$.
\end{theorem}

\begin{corollary}\label{cor:EnokizonoToMSv}
    Let $D$ be a divisor on $S$ such that $D^2>0$ and $\dim H^0(S, \mathcal{O}_S(D)) > \dim H^1(S, \mathcal{O}_S)_n$. Then $D$ is a Miyaoka-Sakai divisor.
\end{corollary}
\begin{proof}
Because $\dim H^0(S, \mathcal{O}_S(D)) > \dim H^1(S, \mathcal{O}_S)_n \ge 0$ and $D^2 > 0$, we know that $D$ is big. 
Let $D = P_{\mathbb{Z}} + N_{\mathbb{Z}}$ be the integral Zariski decomposition and $D = P + N$ the Zariski decomposition of $D$.

By Proposition \ref{prop:IntZD_from_ZD}, $P_{\mathbb{Z}} = P + M$, where $0\le M\le N$. Since $D$ is big, so is $P$ and hence $P_\mathbb{Z}$ is big.

By Lemma \ref{lem:h0D=h0P}, 
$$\dim H^0(S, \mathcal{O}_S(P_{\mathbb{Z}})) = \dim H^0(S, \mathcal{O}_S(D)) > \dim H^1(S, \mathcal{O}_S)_n.$$

By Theorem \ref{thm:Enokizono}, $H^1(S, \mathcal{O}_S(-P_{\mathbb{Z}})) = 0$. Therefore, by Lemma \ref{lem:Miyaoka}, we see that $D$ is a Miyaoka-Sakai divisor.
\end{proof}

In the case that $H^1(S, \mathcal{O}_S)_n =0$, the condition $H^0(S, \mathcal{O}_S(D))>H^1(S, \mathcal{O}_S)_n =0$ is equivalent to that $D$ is effective. With this observation, we will show that Miyaoka-Sakai vanishing holds on smooth Frobenius split surfaces given that $D$ is effective.

\begin{definition}[{\cite[Definition 2]{Mehta1985}, \cite[Definition 3.1]{Schwede2010}}] 
The surface $S$ is \textbf{Frobenius split} ($F$-split) if the Frobenius morphism $\mathcal{O}_S\to F_*\mathcal{O}_S$ splits.
It is \textbf{globally $F$-regular} if there exists a positive integer $e$ such that the natural morphism $\mathcal{O}_S\to F^e_*(\mathcal{O}_S(D))$ splits for every effective divisor $D$.
\end{definition}

It is easy to checked that globally $F$-regular implies $F$-split. However, the converse is not true. For example, ordinary Abelian surfaces are $F$-split but not globally $F$-regular.

\begin{lemma}\label{lem:F-Split=>NoNilpotent}
    If $S$ is Frobenius split, then $H^1(S,\mathcal{O}_S)_n = 0$.
\end{lemma}
\begin{proof}
    Since $S$ is Frobenius split, the Frobenius morphism induces an injection on cohomology: 
$$F^*: H^1(S, \mathcal{O}_S)\to H^1(S, \mathcal{O}_S) = H^1(S, F_*\mathcal{O}_S).$$ 
 By induction, if $(F^*)^e(v)=0$, then $v=0$. Therefore, $H^1(S, \mathcal{O}_S)_n = 0$.
\end{proof}

By \cite[Lemma 1]{Mehta1991}, If $S$ is Frobenius split, then $H^0(S,\mathcal{O}_S((1-p)K_S)) \ne 0$. It follows that a smooth surface $S$ with Kodaira dimension $\kappa(S)\ge 1$ cannot be Frobenius split. Indeed, by \cite[Theorem 4.3]{Schwede2010}, if $S$ is $F$-split, then there exist an effective $\mathbb{Q}$-divisor $\Delta$ such that $(S,\Delta)$ is log Calabi-Yau, that is, $(S,\Delta)$ is log canonical and $K_S+\Delta$ is $\mathbb{Q}$-trivial.

\begin{example}[Frobenius Split Surfaces] Here is an incomplete list of surfaces that are Frobenius split.

\begin{enumerate}
    \item Toric Surfaces. 
    
By \cite[Proposition 6.4]{Smith2000}, a smooth toric surface $S$ is globally $F$-regular and hence Frobenius split. On toric surfaces, the Kawamata-Viehweg vanishing theorem holds under extra conditions (see \cite[Theorem A]{Wang2017} and \cite[Theorem 1.1]{Tanaka2024} and references therein).

\item Weak del Pezzo Surfaces in characteristic $p>5$.

    By \cite[Theorem 1.1]{Kawakami2025}, a smooth weak del Pezzo surface over an algebraically closed field of characteristic $p>5$ is globally $F$-regular, in particular, Frobenius split. 
    
    For log del Pezzo surfaces of characteristic $p>5$, the Kawamata-Viehweg vanishing theorem has been proved (see \cite[Theorem 1.1]{Arvidsson2022}). the Kawamata-viehweg vanishing theorem also holds on del Pezzo surfaces over imperfect ground fields of characteristic $p > 3$ (see \cite[Theorem 1.1]{Das2021}).

\item Ruled surfaces $\mathbb{P}(\mathcal{O}_C\oplus \mathcal{L})$ over a Frobenius split curve $C$.
By \cite[Proposition 3.1]{Gongyo2016}, if $C$ is Frobenius split, that is, $C$ is $\mathbb{P}^1$ or an ordinary elliptic curve, then $S=\mathbb{P}(\mathcal{O}_C\oplus \mathcal{L})$ is Frobenius split.

\item Ordinary abelian surfaces.

    By \cite[Proposition 3.1]{Joshi2003}, an ordinary Abelian surface $S$ is Frobenius split. For the definition of smooth ordinary variety, see \cite[Section 3]{Joshi2003} or \cite[Definition 7.2]{Bloch1986}. We will show in next section that the Kawamata-Viehweg vanishing theorem still holds on Abelian surfaces and hyperelliptic surfaces (see Proposition \ref{prop:AbelianHyperelliptic}.)
    
\item Elliptic surfaces with ordinary base curves and fibers.

    Suppose $S$ admits be an elliptic fibration  $\pi: S\to C$ over an elliptic curve $C$. If the curve $C$ and general fibers are ordinary and $\mathcal{O}_S((1-p)K_S)=\mathcal{O}_S$, then by \cite[Theorem 5.7]{Shirato2017}, $S$ is Frobenius split.

\item Weakly ordinary $K3$ surfaces.

    By \cite[Proposition 2.6]{Patakfalvi2017} (see also  \cite[Proposition 9]{Mehta1985}), a weakly ordinary smooth $K3$ surface, that is, the action of the Frobenius morphism on $H^2(S,\mathcal{O})$ is bijective, is Frobenius split.
\end{enumerate}
\end{example}

We want to remark that \cite[Proposition 3.2]{Kawakami2025} provides sufficient conditions for a $F$-pure klt (including smooth) surfaces to be Frobenius split.

On smooth Frobenius split surfaces, Enokizono's vanishing leads to the following variant of the Kawamata-Viehweg vanishing theorem.

\begin{proposition}[Kawamata-Viehweg Vanishing on Frobenius Split Surfaces]\label{prop:D>=0_KV_FSplit}
   Assume that $S$ is Frobenius split and $D$ is a nef and big $\mathbb{Q}$-divisor on $S$. If $\lceil D \rceil$ is linearly effective, then $H^1(S, \mathcal{O}_S(-\lceil D\rceil)) = 0$.
    In particular, the conclusion holds if $D$ is effective.
\end{proposition}
\begin{proof}
    Since $D$ is nef and big, the divisor $\lceil D\rceil$ is effective, big and $\mathbb{Z}$-positive by \cite[Corollary 3.18]{Enokizono2024}. By Lemma \ref{lem:F-Split=>NoNilpotent}, $\dim H^1(S, \mathcal{O}_S)_n = 0$. By Theorem \ref{thm:Enokizono}, $H^1(S, \mathcal{O}_S(-\lceil D\rceil)) = 0$.
\end{proof}

If $S$ is a smooth toric variety, Proposition \ref{prop:D>=0_KV_FSplit} partially recovers \cite[Theorem B]{Wang2017}.

\begin{proposition}
    Assume that $S=\mathbb{F}_n$ is a Hirzebruch surface and $D$ is a nef and big $\mathbb{Q}$-divisor on $S$. Then 
    $$H^1(S, \mathcal{O}_S(-\lceil D\rceil))=0.$$
\end{proposition}
\begin{proof}
    Since $S$ is rational, $H^1(S, \mathcal{O}_S)=0$. By Theorem \ref{thm:Enokizono}, it suffices to show that $\lceil D\rceil$ is effective. 
    Let $C_0$ be the unique zero section such that $C_0^2=-n$ and $F$ be a fiber. Note that the effective cone is $\mathrm{Eff}(S)=\mathbb{R}_{\ge 0} C_0 + \mathbb{R}_{\ge 0} F$ (see \cite[Theorem 1]{Rosoff2002}) and the Picard group is $\mathrm{Pic}(S) = \mathbb{Z}C_0 + \mathbb{Z}F$. Write $\lceil D\rceil \sim a C_0 + b F$. Since $D$ is nef and big, $k\lceil D\rceil$ is linearly equivalent to an effective $m C_0 + nF$ for some sufficiently larger number $k$. Since $(ka - m)C_0 + (kb -n) F \sim 0$. Intersecting with $F$ and $C_0$ shows that $ka-m = kb-n = 0$ which implies that $a\ge 0$ and $b\ge 0$ and $\lceil D \rceil$ is effective.
\end{proof}

A similar argument also yields a new proof of the Kawamata-Viehweg vanishing theorem on smooth del Pezzo surfaces.

\begin{proposition}[Kawamata-Viehweg Vanishing on Smooth del Pezzo Surfaces {\cite[Theorem 3.1]{Cascini2018}}]
    Assume that $S$ is a del Pezzo surface and $D$ is a nef and big $\mathbb{Q}$-divisor on $S$. Then $H^1(S, \mathcal{O}_S(-\lceil D\rceil))=0$.
\end{proposition}
\begin{proof}
Since $D$ is nef and big, by \cite[Corollary 3.18]{Enokizono2024}, $\lceil D\rceil$ is $\mathbb{Z}$-positive. Therefore, $\lceil D\rceil C > C^2$ for any irreducible curve with $C^2<0$. Since $-K_S$ is ample, by adjunction formula, all irreducible negative curves are $(-1)$-curves. Therefore, $\lceil D\rceil C \ge 0$ for any negative curves. Because $m\lceil D\rceil$ is effective for a sufficiently large $m$, we see that $m\lceil D\rceil C \ge 0$  for any irreducible curve $C$. Therefore, $\lceil D\rceil$ is nef. By Riemann-Roch theorem, and the fact that $\chi(\mathcal{O}_S)=1$, we can conclude that $H^0(S,\mathcal{O}_S(\lceil D\rceil)) \ne 0$.

Since $\lceil D\rceil$ is an effective big $\mathbb{Z}$-positive divisor and $H^1(S, \mathcal{O}_S)=0$, by Theorem \ref{thm:Enokizono}, we obtain $$H^1(S, \mathcal{O}_S(-\lceil D\rceil)) = 0.$$
\end{proof}

\begin{remark}
    The Kawamata-Viehweg vanishing theorem holds for globally $F$-regular pairs (see \cite[Theorem 6.8]{Schwede2010}).
\end{remark}

\begin{proposition}[Miyaoka-Sakai Vanishing on Ruled Surfaces over Frobenius Split Curves]\label{prop:MS_FSplit_Ruled}
 Assume that $S=\mathbb{P}(\mathcal{O}_C\oplus \mathcal{L})\to C$ is a geometrically ruled surface, $D$ is a big divisor on $S$ with $D^2>0$, and $C$ is Frobenius split. If $H^1(S, \mathcal{O}_S(-D))\ne 0$, then $D$ is a Miyaoka-Sakai divisor.
\end{proposition}
\begin{proof}
Since $C$ is Frobenius split, $C$ is $\mathbb{P}^1$ or an ordinary elliptic curve. By \cite[Proposition 3.1]{Gongyo2016}, $S$ is Frobenius split.
Let $D = \mathbb{P}_{\mathbb{Z}} + N_{\mathbb{Z}}$ be the integral Zariski decomposition of $D$, where $P_{\mathbb{Z}}$ is the $\mathbb{Z}$-positive part. Then $P_\mathbb{Z}^2\ge D^2 >0$. Let $a$ and $b$ be integers such that $P_{\mathbb{Z}}$ is numerically equivalent to $aC_0 + b F$, where $C_0$ is the unique curve with $C_0^2 = - n$, $F$ is a fiber, and $n = \deg\mathcal{L}$. Because $P_{\mathbb{Z}}$ is big and $F^2 = 0$, $P_{\mathbb{Z}} F>0$ which implies $a>0$. Because $P_{\mathbb{Z}}^2=(a C_0 + b F)^2 = -n a^2 + 2 a b >0$. Hence $-a n + 2b >0$. Note that $-K_X = 2C_0 +  (n + 2 - 2g) F$. Then 
$$-P_\mathbb{Z} K_S = -a n + 2b + a(2 - 2g) > 0.$$
Note that $\chi(\mathcal{O}_S)\ge 1 - q = 1 - g(c) \ge 0$. By Riemann-Roch theorem, we see that
$$\dim H^0(X, \mathcal{O}_X(P_\mathbb{Z})) \ge \frac{1}{2}(P_\mathbb{Z}^2 - K_S P_\mathbb{Z}) + \chi(\mathcal{O}_S) > 0.$$
We conclude that $P_{\mathbb{Z}}$ is effective. 

By Lemma \ref{lem:F-Split=>NoNilpotent}, $\dim H^0(S, \mathcal{O}_S)_n=0$. Since $P_{\mathbb{Z}}$ is effective, by Theorem \ref{thm:Enokizono}, we see that $H^1(S, \mathcal{O}_S(-P_{\mathbb{Z}})) = 0$. Then $D$ is a Miyaoka-Sakai divisor by Lemma \ref{lem:Miyaoka}. 
\end{proof}

Similar to the proof the Proposition \ref{prop:MS_FSplit_Ruled}, we also have the following Miyaoka-Sakai Vanishing on weak del Pezzo surfaces.
\begin{proposition}[Miyaoka-Sakai Vanishing on Smooth Weak del Pezzo Surfaces]
Assume that $S$ is a weak del Pezzo surface and $D$ is a big divisor on $S$ with $D^2>0$. Then $D$ is a Miyaoka-Sakai divisor.
\end{proposition}
\begin{proof}
Let $D = P_\mathbb{Z} + N_{\mathbb{Z}}$ be the integral Zariski decomposition with $P_{\mathbb{Z}}$ the $\mathbb{Z}$-positive part. Since $S$ is rational, $\dim H^1(S,\mathcal{O}_S)_n \le \dim H^1(S,\mathcal{O}_S) = 0$. It suffices to show that $P_\mathbb{Z}$ is effective. By Riemann-Roch theorem, we get 
$$\dim H^0(S, \mathcal{O}_S(P_\mathbb{Z})) \ge \frac{1}{2}(P_\mathbb{Z}^2 - K_S P_\mathbb{Z}) +\chi(\mathcal{O}_S)\ge \frac12P_\mathbb{Z}^2 + 1>1.$$ 
By Theorem \ref{thm:Enokizono}, $H^1(S, \mathcal{O}_S(-P_{\mathbb{Z}})) = 0$. By Lemma \ref{lem:Miyaoka}, $D$ is a Miyaoka-Sakai divisor.
\end{proof}

In the rest of this section, we show Miyaoka-Sakai vanishing holds on minimal surfaces of Kodaira dimension zero.

\begin{proposition}[Miyaoka-Sakai Vanishing on $K3$ and Enriques Surfaces]
Assume that $S$ is a $K3$ surface or an Enriques' surface as defined in \cite{Bombieri1977}, and $D$ is a big divisor on $S$ with $D^2>0$. Then $D$ is a Miyaoka-Sakai divisor.
\end{proposition}
\begin{proof}
By Lemma \ref{lem:h0D=h0P} and Corollary \ref{cor:EnokizonoToMSv}, it suffices to show that $\dim H^0(S, \mathcal{O}_S(P_{\mathbb{Z}}))> q(S) = \dim H^1(S,\mathcal{O}_S)$.

Since $K_S$ is numerically trivial, by the Riemann-Roch theorem, we have
    $$
   \dim H^0(S, \mathcal{O}_S(P_{\mathbb{Z}}))\ge \frac{1}{2}P_{\mathbb{Z}}^2 + \chi(\mathcal{O}_S) > \chi(\mathcal{O}_S) \ge \dim H^1(S, \mathcal{O}_S)\ge\dim H^1(S, \mathcal{O}_S)_n. 
    $$ 
    Here, the inequality $\chi(\mathcal{O}_S) \ge \dim H^1(S, \mathcal{O}_S)$ follows from the classification theorem (\cite[Theorem 5]{Bombieri1977}).
\end{proof}

\begin{proposition}[Kawamata-Viehweg Vanishing on Abelian and Hyperelliptic Surfaces]\label{prop:AbelianHyperelliptic}
     Suppose that $S$ is an Abelian surface or a hyperelliptic surface and $D$ is a big divisor on $S$ with $D^2>0$. Then $H^1(S, \mathcal{O}_S(-D))=0$. 
     Moreover, $H^1(S, \mathcal{O}_S(-\lceil M \rceil)) = 0$ for any nef and big $\mathbb{Q}$-divisor $M$ on $S$.
\end{proposition}

\begin{proof}
     From the classification theorem (\cite[Theorem 6 and the Proposition below it]{Bombieri1977}), we know that
     $K_S$ is numerically trivial and $\chi(\mathcal{O}_S)=0$. Since $D$ is big and $D^2 > 0$, by the Riemann-Roch theorem, we see that
    $$
    \dim H^0(S, \mathcal{O}_S(D)) = \frac{1}{2}D^2 + \dim H^1(S,\mathcal{O}_S(D)) >0.
    $$
    Therefore, $D$ is effective.

    From the classification theorem (\cite[Theorem 4]{Bombieri1977}), we know that there is an unramified finite covering: $\pi: A \to S$ where $A$ is an Abelian surface. Because $A$ is Abelian, there is no negative definite curves on it. Therefore, any effective divisor is nef.
    For any irreducible curve $C$ on $S$, we have $(\deg \pi)C^2 = (\pi^*(C))^2 \ge 0$. Therefore, the effective divisor $D$ on $S$ is also nef.
    
    Because $D$ is nef and big, by \cite[Theorem 3]{Mukai2013} (see also Corollary \ref{cor:MRVanishing}), we obtain $H^1(S, \mathcal{O}_S(-D))=0$.

    Write the roundup of $M$ as $\lceil M \rceil = M + \Delta.$  Because $\Delta$ is effective, it is then nef. Consequently, $\lceil M\rceil$ is nef and big. Therefore, $H^1(S, \mathcal{O}_S(-\lceil M\rceil))=0$.
\end{proof}

\section{Weak Miyaoka-Sakai theorem\label{sec:MiyaokaSakai=MumfordRamanujam}}

Let $F: S \to S$ be the Frobenius morphism. The natural derivation $d : \mathcal{O}_S \to \Omega_S$ induces a morphism $F_*d: F_*\mathcal{O}_S\to F_*\Omega_S$. We denote by $\mathcal{B}_S^1$ the image of $F_*d$. By \cite[Proposition 3]{Tango1972}), we have the following short exact sequence
\begin{equation}\label{eq:deRham_augment}
\begin{tikzcd}[column sep = 2em]
0 \arrow[r] & \mathcal{O}_S \arrow[r] & F_*\mathcal{O}_S \arrow[r] & \mathcal{B}_S \arrow[r] & 0.
\end{tikzcd}
\end{equation}

For any divisor $D$ on $S$, twisting the exact sequence \eqref{eq:deRham_augment} by $\mathcal{O}_S(D)$ and taking cohomology, we obtain the following long exact 
sequence:
\begin{equation}\label{eq:Frobenius}
\begin{tikzcd}[column sep = 2em]
0 \arrow[r] &
H^0(S,\mathcal{O}_S(-D)) \arrow[r] &
H^0 \big(S,F_*(\mathcal{O}_S(-pD))\big) \arrow[r] &
H^0 \big(S,\mathcal{B}_S(-D)\big)
\arrow[from=1-4, to=2-2, rounded corners,
    to path={ 
    -- ([xshift=2ex]\tikztostart.east)
    |- ([yshift=-2ex]\tikztostart.south) 
    -| ([xshift=-2ex]\tikztotarget.west) node[pos=0.25, fill=white, inner sep=1pt] {$\delta$}
    -- (\tikztotarget)
}]
\\
&
H^1(S,\mathcal{O}_S(-D)) \arrow[r] &
H^1 \big(S,F_*(\mathcal{O}_S(-pD))\big) \arrow[r] &
H^1 \big(S,\mathcal{B}_S(-D)\big),
\end{tikzcd}
\end{equation}
where $\delta$ is the coboundary
map. We denote by $K(D)$ the kernel of $F_*: H^1(S, \mathcal{O}_S(-D)\to H^1(S, F_*(\mathcal{O}_S(-p D))$. Then $K(D) = \delta(H^0(S, \mathcal{B}_S(-D)))$. If $D$ is big divisor or effective, then $$H^0(S, F_*(\mathcal{O}_S(-pD)))=H^0(S, \mathcal{O}_S(-pD)) = 0.$$ Therefore, $\ker \delta = 0$ and $$K(D) = H^0(S, \mathcal{B}_S(-D)).$$

From the long exact sequence \eqref{eq:Frobenius}, 
we see that $H^1(S, \mathcal{O}_S(-D)) = 0$ 
if $H^1(S, \mathcal{O}_S(-p^e D)) = 0$ for some positive integer $e$ 
and $H^0(S, \mathcal{B}_S(-p^kD))=0$ for all $0\le k\le e-1$. 

For a divisor $D$, in \cite[Section 1.2]{Mukai2013}, Mukai shows that 
$$\mathcal{B}_S(-D) = F_*(d Q(S)\cap \Omega_S(-pD)),$$
where $Q(S)$ is the function field of $S$.
Consequently,
\begin{equation}\label{eq:MukaiH^0B_S}
H^0(S, \mathcal{B}_S(-D)) = \{d f \in \Omega_{Q(S)}\mid f\in Q(S), d f\in H^0(S, \Omega_S(-p D))\}.
\end{equation}
Using this description of $H^0(S, \mathcal{B}_S(-D))$, Mukai shows that $H^0(S, \mathcal{B}_S(-D))=0$ when $D$ is nef and big which leads to a proof of Mumford-Ramanujam vanishing theorem \cite[Theorem 3]{Mukai2013}. In the following, we will show that $H^0(S, \mathcal{B}_S(-D))=0$ for any big divisor $D$ with $D^2>0$. This vanishing leads to a weak version of Miyaoka-Sakai theorem which also implies Mumford-Ramanujam vanishing theorem.

We first recall the following lemma of Sakai which still holds in positive characteristic.

\begin{lemma}[{\cite[Lemma 3]{Sakai1990}}]\label{lem:Sakai-big}
    Let $\pi: S \to X$ be a birational morphism. Let $D$ be a divisor on $S$, and set $D_X=\pi_*D$.
    \begin{enumerate}
        \item If $D$ is big, then $D_X$ is big.
        \item If $D^2>0$ and $D_X$ is big, then $D$ is big.
    \end{enumerate}
\end{lemma}

\begin{lemma}\label{lem:mukai}
Assume that $S$ is a (relative) minimal surface of Kodaira dimension $\kappa(S)\le 1$ but not quasi-elliptic if $\kappa(S) = 1$. Let $D$ be a big divisor on $S$ with $D^2>0$. Then
    $$H^0(S, \mathcal{B}_S(-D)) = 0.$$
\end{lemma}
\begin{proof}
From the definition of $\mathcal{B}_S$, we know that
$$H^0(S, \mathcal{B}_S(-D))\subseteq H^0(S, \Omega_S(-pD)).$$

By the classification of algebraic surfaces (see \cite{Bombieri1977}), we know that $X$ is the projective plane, a ruled surface, an Abelian surface, a $K3$ surface, an Enriques surface, a hyperelliptic surface, a quasi-hyperelliptic surface, or an elliptic surface.

     \begin{enumerate}
     \item If $S$ is the projective plane, then $H^0(S,\Omega_S(-pD))=0$ from the twisted Euler sequence. Therefore, $H^0(S, \mathcal{B}_S(-D))=0$. 
        \item Suppose $S$ is ruled, elliptic, hyperelliptic, or quasi-hyperelliptic. Then it admits a fibration $f: S\to C$ whose fibers has genus $0$ if it is rule or $1$ otherwise. 
        Denote by $F$ a general fiber. 
        Consider the exact sequence $$0\to f^*\Omega_C\to \Omega_S\to \Omega_{S/C}\to 0.$$ 
        Since $F$ is smooth of genus at most $1$, $\deg\Omega_{S/C}|_F = \deg\Omega_F \le 0$.
        If 
        $$H^0(S, f^*\Omega_C(-D))\neq 0 \quad \text{or} \quad H^0(S, \Omega_{S/C}(-D))\neq 0,$$ 
        then $$\deg (f^*\Omega_C(-D))|_F = -DF\ge 0 \quad \text{or} \quad\deg (\Omega_{S/C}(-D))|_F = \deg\Omega_F -DF\ge 0.$$ It follows that $D F\le 0$. 
        Since $D$ is big, there is an effective $\mathbb{Q}$-divisor $E$ and an ample $\mathbb{Q}$-divisor $A$ such that $D = A + E$. Since $F$ is a fiber, we know that $F$ is nef. Therefore, $DF = A F + E F > 0$ which leads to a contradiction. Therefore, $H^0(S, \Omega_S(-D))=0$ which implies the desired conclusion.

    \item If $S$ is Abelian, then $\Omega_S$ is trivial. If $H^0(S, \Omega_S(-D))\ne 0$, then $H^0(S, \mathcal{O}_S(-D))\ne 0$. It follows that $HD\le 0$ for an ample divisor $H$. By Hodge index theorem, we get $D^2\le 0$ which contradicts the assumption that $D^2>0$. Therefore, $H^0(S, \Omega_S(-D))=0$ which again leads to the desired conclusion.

\item Suppose $S$ is $K3$ or Enriques. If $H^0(S, \mathcal{B}_S(-D))\ne 0$, then a nonzero section $s$ defines a nontrivial morphism $\phi: \mathcal{O}_S(D)\to \mathcal{B}_S$ by multiplying $s$. Since $\mathcal{B}_S$ is torsion-free sheaf, the image $\operatorname{im}(\phi)$ as a subsheaf of $\mathcal{B}_S$ is also torsion-free. This forces $\ker(\phi)=0$. Therefore, $\phi$ is injective and 
$$\dim H^0(S,\mathcal{B}_S)\ge \dim H^0(S, \mathcal{O}_S(D)).$$
Since $K_S$ is numerically trivial, by Riemann-Roch Theorem, we have
$$\dim H^0(S, \mathcal{O}_S(D))\ge \frac{1}{2}D^2 + \chi(\mathcal{O}_S) - \dim H^0(S, \mathcal{O}_S(K_S-D)).$$
Because $-(K_S-D)\equiv D$ is big, we must have $H^0(S, \mathcal{O}_S(K_S-D)) = 0$.
By \cite{Bombieri1977}, we know that $q< \chi(\mathcal{O}_S)$. Therefore,
$$\dim H^0(S, \mathcal{B}_S) \ge \dim H^0(S, \mathcal{O}_S(D)) > q = \dim H^1(S, \mathcal{O}_S).$$
However, by \cite[Lemma 1.14]{Mukai2013},
$H^0(S, \mathcal{B}_S) = \ker\big[H^1(S, \mathcal{O}_S)\to H^1(S, F_*\mathcal{O}_S)\big]$. That is a contradiction. We conclude that $H^0(S, \mathcal{B}_S(-D))= 0$. 
\end{enumerate}
\end{proof}

The following lemma is derived from the proof of \cite[Proposition 3.2]{Mukai2013}.

\begin{lemma}\label{lem:Mukai_VanishingB_S}
    Let $\pi: S\to X$ be a contraction of a $(-1)$-curve $E$ and $D$ be a 
    divisor on $S$. Set $D_X = \pi_*D$. If $H^0(X, \mathcal{B}_X(-D_X)) = 0$, then
    $H^0(S, \mathcal{B}_S(-D)) =0$. 
\end{lemma}

\begin{proof}
Since $\pi$ is a birational morphism, the functional fields $Q(S)$ and $Q(X)$ can be identified as the same. 
By the equality \eqref{eq:MukaiH^0B_S}, it suffices to show that $H^0(S, \Omega_S(-pD)) \subseteq H^0(X, \Omega_X(-p D_X))$.

Note that $D = \pi^*D_X - a E$, for an integer $a$. Consider the short exact sequence
\begin{equation}
\label{eq:KahlerDifferentialSequence}
    0\to \pi^*\Omega_X \to \Omega_S \to i_*\Omega_{E}= i_*\mathcal{O}_E(-2)\to 0
\end{equation}

    If $a\ge 0$,
twisting the exact sequence \eqref{eq:KahlerDifferentialSequence} by $\mathcal{O}_S(-pD)$
 and then taking cohomology, we get the short exact sequence
    $$0\to H^0(S, \pi^*\Omega_X(-p(\pi^*D_X - a E)))\to H^0(S, \Omega_S(-p D)) \to H^0(E, \mathcal{O}_E(-2 - p a)).$$
Because $a\ge 0$, $H^0(E, \mathcal{O}_E(-2 - p a)) = 0$ which implies 
    $$
    H^0(S, \pi^*\Omega_X({-p(\pi^*D_X - a E))}) = H^0(S, \Omega_S(-p D)).
    $$
    By the projection formula,
    $$H^0(S, \Omega_S(-p D)) = H^0(S, \pi^*\Omega_X({-p(\pi^*D_X - a E))}) = H^0(X, \Omega_X(-p D_X)).$$

    If $a < 0$, twisting the exact sequence \eqref{eq:KahlerDifferentialSequence} by $\mathcal{O}_S(-p\pi^*D_X)$
 and then taking the cohomology, we get the short exact sequence
    $$0\to H^0(S, \pi^*\Omega_X(-p(\pi^*D_X)))\to H^0(S, \Omega_S(-p (D + a E))) \to H^0(E, \mathcal{O}_E(-2)) 
    $$
    which implies 
    $$
    H^0(S, \pi^*\Omega_X({-p \pi^*D_X))}) = H^0(S, \Omega_S(-p (D + aE))).
    $$
    Because $a < 0$, 
    $$H^0(S, \Omega_S(-p D))\subseteq H^0(S, \Omega_S(-p (D + aE))).$$
    By the projection formula, we have $H^0(S, \pi^*\Omega_X({-p \pi^*D_X))}) = H^0(X, \Omega_X(-p D_X)))$. 
    Therefore, 
    $$H^0(S, \Omega_S(-p D))\subseteq H^0(X, \Omega_X(-p D_X)) .$$
\end{proof}

\begin{corollary}[Generalized Mumford-Ramanujam Vanishing]\label{cor:generalizedMumford-Ramanujam}
    Let $D$ be a big divisor on $S$ such that $D^2>0$.
    Assume that $\kappa(S)\le 1$ and $S$ is not quasi-elliptic if $\kappa(S) =1$. 
    
    If $H^1(S, \mathcal{O}_S(-p^e D))=0$ for a nonnegative integer $e$, then $H^1(S, \mathcal{O}_S(-D))=0.$
\end{corollary}

\begin{proof}
Let $X$ be a minimal model of $S$.
Write $D_X = \pi_* D$ and $D = \pi^*D_X - E$, where $E$ is an exceptional divisor. Then $D_X^2 = D^2 - E^2 \ge D^2 >0$. By Lemma \ref{lem:Sakai-big}, $D_X$ is also big. Therefore, by Lemma \ref{lem:mukai}, we see that $H^0(X, \mathcal{B}_X(-p^kD_X))=0$ for any nonnegative integer $k$.

By Lemma \ref{lem:Mukai_VanishingB_S}, $$H^0(S, \mathcal{B}_S(-p^kD))=0$$ for any nonnegative integer $k$. It follows that the map
$$
\begin{tikzcd}[column sep = 2em]
    H^1(S, \mathcal{O}_S(-p^k D)) \arrow[r, hookrightarrow] & H^1(S, \mathcal{O}_S(-p^{k + 1} D))
\end{tikzcd}
$$ is injective for each nonnegative integer. 

Since $H^1(S, \mathcal{O}_S(-p^e D))=0$ for a nonnegative integer $e$, by induction, we obtain $H^1(S, \mathcal{O}_S(-D))=0$.
\end{proof}

\begin{corollary}[Mumford-Ramanujam Vanishing ({\cite[Section 3]{Mukai2013}})]\label{cor:MRVanishing}
    Let $D$ be a nef and big divisor on $S$.
    Assume that $\kappa(S)\le 1$ and $S$ is not quasi-elliptic if $\kappa(S) =1$.
    Then $H^1(S, \mathcal{O}_S(-D))=0$.
\end{corollary}
\begin{proof}
By \cite[Proposition 2.1]{Szpiro1979} (see also \cite[Proposition 2]{LewinMenegaux1981}), we know that $H^1(S, \mathcal{O}_S(-kD))=0$ for all $k  \gg  0$, 
in particular, $H^1(S, \mathcal{O}_S(-p^eD))=0$ for some positive integer $e$. 
By Corollary \ref{cor:generalizedMumford-Ramanujam}, we see that $H^1(S, \mathcal{O}_S(-D))=0$.
\end{proof}

Corollary \ref{cor:generalizedMumford-Ramanujam} together with Theorem \ref{thm:Enokizono} implies the following $Q$-divisor variant of the Miyaoka-Sakai theorem.

\begin{corollary}[Weak Miyaoka-Sakai theorem]\label{cor:weakMS}
    Let $D$ be a big divisor on $S$.
    Assume that $\kappa(S)\le 1$ and $S$ is not quasi-elliptic when $\kappa(S) =1$.
    If $H^1(S, \mathcal{O}_S(-D))\ne 0$, then for any sufficiently large integer $e$, 
    there exists a nonzero effective divisor $B_e$ such that
    \begin{enumerate}
        \item $(p^e D - B_e)C\le 0$ for any irreducible component $C$ of $B_e$,
        \item $H^1(S, \mathcal{O}_S(- p^eD + B_e))=0$,
        \item $p^e D-2B_e$ is big if $D^2>0$.
    \end{enumerate}
\end{corollary}
\begin{proof}
Let $e$ be a sufficiently large number such that $\dim H^0(S, \mathcal{O}_S(p^e D)) >\dim H^1(S, \mathcal{O}_S)_n$.
 Let $p^e D = P_{\mathbb{Z}} + N_{\mathbb{Z}}$ be the integral Zariski decomposition of $p^e D$. Since $D^2 > 0$, then $P_{\mathbb{Z}}^2 > D^2 > 0$.
 By Proposition \ref{prop:IntZD_from_ZD}, we see that $P_{\mathbb{Z}}$ is big. 
 Because 
 $$\dim H^0(S, \mathcal{O}_S(P_{\mathbb{Z}}) = \dim H^0(S, \mathcal{O}_S(p^e D)) > \dim H^1(S, \mathcal{O}_S)_n,$$ by \cite[Theorem] {Enokizono2023}, 
 $$H^1(S, \mathcal{O}_S(- P_{\mathbb{Z}})) = 0.$$
    By Lemma \ref{lem:Miyaoka}, there exists a nonzero effective divisor $B_e$ such that 
    \begin{enumerate}
        \item $(D - B_e)C\le 0$ for any irreducible component of $B_e$.
        \item $(D - B_e)^2 \ge D^2$.
        \item $H^1(S, \mathcal{O}_S(-D + B_e))=0$.
    \end{enumerate}
\end{proof}

\begin{corollary}[Ramanujam Vanishing Theorem]\label{cor:RamanujamVanishing}
    Let $D$ be a big divisor on $S$ such that $D^2>0$.
    Assume that $\kappa(S)\le 1$ and $S$ is not quasi-elliptic when $\kappa(S) =1$. If $D$ is numerically connected, then $H^1(S,\mathcal{O}_S(-D))=0$.
\end{corollary}
\begin{proof}
    We proof by contradiction. Assume that $H^1(S,\mathcal{O}_S(-D))\ne 0$. By Corollary \ref{cor:weakMS}, there is a decomposition $D= A + B$ such that $B$ is nonzero and effective, $AB\le 0$ and $A-B$ is big, hence pseudoeffective. However, this contradicts with the numerically connectedness of $D$.
\end{proof}

By Ramanujam's connectedness lemma \ref{lem:Nef=>NumCon}, Mumford-Ramanujam vanishing theorem for nef and big divisors follows from Ramanujam vanishing theorem.


\begin{thebibliography}{{She}91}

\bibitem[ABL22]{Arvidsson2022}
Emelie Arvidsson, Fabio Bernasconi, and Justin Lacini.
\newblock On the {{Kawamata}}--{{Viehweg}} vanishing theorem for log del {{Pezzo}} surfaces in positive characteristic.
\newblock {\em Compositio Mathematica}, 158(4):750--763, April 2022.

\bibitem[B{\u a}d01]{Badescu2001}
Lucian~Silvestru B{\u a}descu.
\newblock {\em Algebraic {{Surfaces}}}.
\newblock Universitext. Springer-Verlag, New York, 2001.

\bibitem[Ber21]{Bernasconi2021}
Fabio Bernasconi.
\newblock Kawamata-viehweg vanishing fails for log del pezzo surfaces in characteristic 3.
\newblock {\em J. Pure Appl. Algebra}, 225(11):106727, 2021.

\bibitem[BK86]{Bloch1986}
Spencer Bloch and Kazuya Kato.
\newblock {$p$-adic etale cohomology}.
\newblock {\em Publications Math{\'e}matiques de l'IH{\'E}S}, 63:107--152, 1986.

\bibitem[BM77]{Bombieri1977}
E.~Bombieri and D.~Mumford.
\newblock Enriques' {{Classification}} of {{Surfaces}} in {{Char}}. p, {{II}}.
\newblock In T.~Shioda and W.~L.~Jr Baily, editors, {\em Complex {{Analysis}} and {{Algebraic Geometry}}: {{A Collection}} of {{Papers Dedicated}} to {{K}}. {{Kodaira}}}, pages 23--42. Cambridge University Press, Cambridge, 1977.

\bibitem[Bog79]{Bogomolov1979}
F.~A. Bogomolov.
\newblock Holomorphic {{Tensors}} and {{Vector Bundles}} on {{Projective Varieties}}.
\newblock {\em Mathematics of the USSR-Izvestiya}, 13(3):499, June 1979.

\bibitem[Che21]{Chen2021}
Yen-An Chen.
\newblock Fujita's conjecture for quasi-elliptic surfaces.
\newblock {\em Mathematische Nachrichten}, 294(11):2096--2104, 2021.

\bibitem[CT18]{Cascini2018}
Paolo Cascini and Hiromu Tanaka.
\newblock Smooth rational surfaces violating {K}awamata-{V}iehweg vanishing.
\newblock {\em Eur. J. Math.}, 4(1):162--176, 2018.

\bibitem[Das21]{Das2021}
Omprokash Das.
\newblock Kawamata-{{Viehweg Vanishing Theorem}} for {{Del Pezzo Surfaces Over Imperfect Fields}} of {{Characteristic P}} {$>$} 3.
\newblock {\em Osaka Journal of Mathematics}, 58(2):477--486, April 2021.

\bibitem[DC13]{DiCerbo2013}
Gabriele Di~Cerbo.
\newblock A cohomological interpretation of {{Bogomolov}}'s instability.
\newblock {\em Proceedings of the American Mathematical Society}, 141(9):3049--3053, June 2013.

\bibitem[DCF15]{DiCerbo2015}
Gabriele Di~Cerbo and Andrea Fanelli.
\newblock Effective {{Matsusaka}}'s theorem for surfaces in characteristic {\emph{p}}.
\newblock {\em Algebra \& Number Theory}, 9(6):1453--1475, September 2015.

\bibitem[Eke88]{Ekedahl1988}
Torsten Ekedahl.
\newblock Canonical models of surfaces of general type in positive characteristic.
\newblock {\em Publications Math\'ematiques de l'IH\'ES}, 67:97--144, 1988.

\bibitem[EL93]{Ein1993}
Lawrence Ein and Robert Lazarsfeld.
\newblock Global generation of pluricanonical and adjoint linear series on smooth projective threefolds.
\newblock {\em Journal of the American Mathematical Society}, 6(4):875--903, 1993.

\bibitem[Eno23]{Enokizono2023}
Makoto Enokizono.
\newblock Vanishing theorems and adjoint linear systems on normal surfaces in positive characteristic.
\newblock {\em Pacific Journal of Mathematics}, 324(1):71--110, June 2023.

\bibitem[Eno24]{Enokizono2024}
Makoto Enokizono.
\newblock An integral version of zariski decompositions on normal surfaces.
\newblock {\em European Journal of Mathematics}, 10(2):38, June 2024.

\bibitem[FdB95]{Fernandez1995}
Guillermo Fern{\'a}ndez~del Busto.
\newblock Bogomolov instability and {Kawamata}-{Viehweg} vanishing.
\newblock {\em J. Algebr. Geom.}, 4(4):693--700, 1995.

\bibitem[GT16]{Gongyo2016}
Yoshinori Gongyo and Shunsuke Takagi.
\newblock Surfaces of globally {{F-regular}} and {{F-split}} type.
\newblock {\em Mathematische Annalen}, 364(3):841--855, April 2016.

\bibitem[GZZ22]{Gu2022}
Yi~Gu, Lei Zhang, and Yongming Zhang.
\newblock Counterexamples to fujita's conjecture on surfaces in positive characteristic.
\newblock {\em Adv. Math.}, 400:Paper No. 108271, 17, 2022.

\bibitem[JR03]{Joshi2003}
Kirti Joshi and C.~S. Rajan.
\newblock Frobenius splitting and ordinarity.
\newblock {\em International Mathematics Research Notices}, 2003(2):109--121, January 2003.

\bibitem[Kaw97]{Kawamata1997}
Yujiro Kawamata.
\newblock On {F}ujita's freeness conjecture for {$3$}-folds and {$4$}-folds.
\newblock {\em Math. Ann.}, 308(3):491--505, 1997.

\bibitem[Kle74]{Kleiman1974}
Steven~L. Kleiman.
\newblock The transversality of a general translate.
\newblock {\em Compos. Math.}, 28:287--297, 1974.

\bibitem[KM98]{Kawachi1998}
Takeshi Kawachi and Vladimir Ma{\c{s}}ek.
\newblock Reider-type theorems on normal surfaces.
\newblock {\em J. Algebr. Geom.}, 7(2):239--249, 1998.

\bibitem[Kos23]{Koseki2023}
Naoki Koseki.
\newblock On the {{Bogomolov}}--{{Gieseker Inequality}} in {{Positive Characteristic}}.
\newblock {\em International Mathematics Research Notices}, 2023(24):20784--20811, December 2023.

\bibitem[KT25]{Kawakami2025}
Tatsuro Kawakami and Hiromu Tanaka.
\newblock Global f-regularity for weak del pezzo surfaces.
\newblock {\em Forum of Mathematics, Sigma}, 13:e77, 2025.

\bibitem[Lan16]{Langer2016}
Adrian Langer.
\newblock The {{Bogomolov}}--{{Miyaoka}}--{{Yau}} inequality for logarithmic surfaces in positive characteristic.
\newblock {\em Duke Mathematical Journal}, 165(14):2737--2769, October 2016.

\bibitem[Lan22]{Langer2022a}
Adrian Langer.
\newblock On boundedness of semistable sheaves.
\newblock {\em Documenta Mathematica}, 27:1--16, 2022.

\bibitem[Laz97]{Lazarsfeld1997}
Robert Lazarsfeld.
\newblock Lectures on linear series. {With} the assistance of {Guillermo} {Fern{\'a}ndez} del {Busto}.
\newblock In {\em Complex algebraic geometry. Lectures of a summer program, Park City, UT, 1993}, pages 163--219. Providence, RI: American Mathematical Society, 1997.

\bibitem[Laz04]{Lazarsfeld2004}
Robert Lazarsfeld.
\newblock {\em Positivity in Algebraic Geometry. {{I}}. {{Classical}} Setting: Line Bundles and Linear Series}, volume~48 of {\em Ergeb. {{Math}}. {{Grenzgeb}}., 3. {{Folge}}}.
\newblock Berlin: Springer, 2004.

\bibitem[LM81]{LewinMenegaux1981}
R.~Lewin-Ménégaŭx.
\newblock Un théorème d’annulation en caractéristique positive.
\newblock In {\em Séminaire sur les pinceaux de courbes de genre au moins deux}, volume~86 of {\em Astérisque}, pages 35--43. Société Mathématique de France, 1981.

\bibitem[LT25]{Larsen2025}
Anne Larsen and Anda Tenie.
\newblock Reider-{{Type Theorems}} on {{Normal Surfaces}} via {{Bridgeland Stability}}.
\newblock {\em International Mathematics Research Notices}, 2025(16):253, August 2025.

\bibitem[Miy80]{Miyaoka1980}
Yoichi Miyaoka.
\newblock On the {Mumford}-{Ramanujam} vanishing theorem on a surface.
\newblock Journ{\'e}es de g{\'e}om{\'e}trie alg{\'e}brique, {Angers}/{France} 1979, 239-247 (1980)., 1980.

\bibitem[MR85]{Mehta1985}
V.~B. Mehta and A.~Ramanathan.
\newblock Frobenius {{Splitting}} and {{Cohomology Vanishing}} for {{Schubert Varieties}}.
\newblock {\em Annals of Mathematics}, 122(1):27--40, 1985.

\bibitem[MS91]{Mehta1991}
V.~B Mehta and V~Srinivas.
\newblock Normal {{F-pure}} surface singularities.
\newblock {\em Journal of Algebra}, 143(1):130--143, October 1991.

\bibitem[Muk13]{Mukai2013}
Shigeru Mukai.
\newblock Counterexamples to {K}odaira's vanishing and {Y}au's inequality in positive characteristics.
\newblock {\em Kyoto J. Math.}, 53(2):515--532, 2013.

\bibitem[Mum78]{Mumford1978}
D.~Mumford.
\newblock Some footnotes to the work of {C}. {P}. {Ramanujam}.
\newblock C.P. {Ramanujam}. - {A} tribute. {Collect}. {Publ}. of {C}.{P}. {Ramanujam} and {Pap}. in his {Mem}., {Tata} {Inst}. fundam. {Res}., {Stud}. {Math}. 8, 247-262 (1978)., 1978.

\bibitem[Nak93]{Nakashima1993b}
Tohru Nakashima.
\newblock {On Reider's method for surfaces in positive characterstic.}
\newblock {\em Journal f{\"u}r die reine und angewandte Mathematik}, 438:175--186, 1993.

\bibitem[PST17]{Patakfalvi2017}
Zsolt Patakfalvi, Karl Schwede, and Kevin Tucker.
\newblock Positive characteristic algebraic geometry, January 2017.

\bibitem[Ram74]{Ramanujam1974}
C.~P. Ramanujam.
\newblock Supplement to the article ''{Remarks} on the {Kodaira} vanishing theorem''.
\newblock {\em J. Indian Math. Soc., New Ser.}, 38:121--124, 1974.

\bibitem[Ray78]{Raynaud1978}
M.~Raynaud.
\newblock Contre-exemple au ``vanishing theorem''\ en caract\'eristique {$p>0$}.
\newblock In {\em C. {P}. {R}amanujam---a tribute}, volume~8 of {\em Tata Inst. Fundam. Res. Stud. Math.}, pages 273--278. Springer, Berlin-New York, 1978.

\bibitem[Rei88]{Reider1988}
Igor Reider.
\newblock Vector bundles of rank {$2$} and linear systems on algebraic surfaces.
\newblock {\em Ann. of Math. (2)}, 127(2):309--316, 1988.

\bibitem[Ros02]{Rosoff2002}
Jeff Rosoff.
\newblock Effective divisor classes on a ruled surface.
\newblock {\em Pacific Journal of Mathematics}, 202(1):119--124, January 2002.

\bibitem[Sak84]{Sakai1984}
Fumio Sakai.
\newblock Weil divisors on normal surfaces.
\newblock {\em Duke Mathematical Journal}, 51(4):877--887, December 1984.

\bibitem[Sak90]{Sakai1990}
Fumio Sakai.
\newblock Reider-serrano's method on normal surfaces.
\newblock In Andrew~John Sommese, Aldo Biancofiore, and Elvira~Laura Livorni, editors, {\em Algebraic Geometry}, pages 301--319, Berlin, Heidelberg, 1990. Springer Berlin Heidelberg.

\bibitem[{She}91]{Shepherd-Barron1991}
N.~I. {Shepherd-Barron}.
\newblock Unstable vector bundles and linear systems on surfaces in characteristic $p$.
\newblock {\em Inventiones mathematicae}, 106(1):243--262, December 1991.

\bibitem[Shi17]{Shirato2017}
Tomoaki Shirato.
\newblock Frobenius splitting for some {{Abelian}} fiber spaces.
\newblock {\em Journal of Pure and Applied Algebra}, 221(6):1383--1406, June 2017.

\bibitem[Smi00]{Smith2000}
Karen~E. Smith.
\newblock Globally {{F-regular}} varieties: Applications to vanishing theorems for quotients of {{Fano}} varieties.
\newblock {\em Michigan Mathematical Journal}, 48(1), 2000.

\bibitem[Spr96]{Spreafico1996}
Maria~Luisa Spreafico.
\newblock Bertini type theorems for vector bundles in any characteristic.
\newblock {\em Commun. Algebra}, 24(13):4147--4157, 1996.

\bibitem[SS10]{Schwede2010}
Karl Schwede and Karen~E. Smith.
\newblock Globally {{F-regular}} and log {{Fano}} varieties.
\newblock {\em Advances in Mathematics}, 224(3):863--894, June 2010.

\bibitem[Szp79]{Szpiro1979}
L.~Szpiro.
\newblock Sur le théorème de rigidité de parsin et arakelov.
\newblock In {\em Astérisque}, volume~64, pages 169--202. Société Mathématique de France, 1979.

\bibitem[Tan72]{Tango1972}
Hiroshi Tango.
\newblock On the behavior of extensions of vector bundles under the {{Frobenius}} map.
\newblock {\em Nagoya Mathematical Journal}, 48:73--89, 1972.

\bibitem[Tan24]{Tanaka2024}
Hiromu Tanaka.
\newblock Kawamata-{{Viehweg}} vanishing for toric varieties, October 2024.

\bibitem[Ter98]{Terakawa1998}
Hiroyuki Terakawa.
\newblock On the {{Kawamata-Viehweg}} vanishing theorem for a surface in positive characteristic.
\newblock {\em Archiv der Mathematik}, 71(5):370--375, November 1998.

\bibitem[Ter99]{Terakawa1999}
Hiroyuki Terakawa.
\newblock The d-very ampleness on a projective surface in positive characteristic.
\newblock {\em Pacific Journal of Mathematics}, 187(1):187--199, January 1999.

\bibitem[Tot19]{Totaro2019}
Burt Totaro.
\newblock The failure of {K}odaira vanishing for {F}ano varieties, and terminal singularities that are not {C}ohen-{M}acaulay.
\newblock {\em J. Algebraic Geom.}, 28(4):751--771, 2019.

\bibitem[WX17]{Wang2017}
Yuan Wang and Fei Xie.
\newblock Vanishing on toric surfaces, July 2017.

\bibitem[Xie10]{Xie2010}
Qihong Xie.
\newblock Counterexamples to the {K}awamata-{V}iehweg vanishing on ruled surfaces in positive characteristic.
\newblock {\em J. Algebra}, 324(12):3494--3506, 2010.

\bibitem[Xie11]{Xie2011}
Qihong Xie.
\newblock A characterization of counterexamples to the {K}odaira-{R}amanujam vanishing theorem on surfaces in positive characteristic.
\newblock {\em Chinese Ann. Math. Ser. B}, 32(5):741--748, 2011.

\bibitem[YZ20]{Ye2020}
Fei Ye and Zhixian Zhu.
\newblock On {F}ujita's freeness conjecture in dimension 5.
\newblock {\em Adv. Math.}, 371:107210, 56, 2020.
\newblock With an appendix by Jun Lu.

\bibitem[Zha23]{Zhang2023}
Yongming Zhang.
\newblock The {\emph{d}}-very ampleness of adjoint line bundles on quasi-elliptic surfaces.
\newblock {\em Journal of Pure and Applied Algebra}, 227(5):107271, May 2023.

\end{thebibliography}
\end{document}